\documentclass[12pt,a4paper]{amsart}

\usepackage{amsmath,amsthm,amsfonts,amssymb,array,amscd}
\usepackage{amsmath,geometry,amssymb,amsfonts,amsthm,graphicx,enumerate,latexsym,tabularx,amscd}
\usepackage{graphicx,color}
\usepackage[all,cmtip]{xy}

\usepackage{refcheck} 
\norefnames
\nocitenames
\usepackage{amsthm}
 \newtheorem{thm}{Theorem}[section]
 \newtheorem{prop}[thm]{Proposition}
 \newtheorem{lem}[thm]{Lemma}
 \newtheorem{cor}[thm]{Corollary}
\theoremstyle{definition}
 \newtheorem{exm}[thm]{Example}
 \newtheorem{dfn}[thm]{Definition}
\theoremstyle{remark}
 \newtheorem{rem}[thm]{Remark}
 \numberwithin{equation}{section}
\theoremstyle{definition}
\theoremstyle{remark}
 \numberwithin{equation}{section}

\renewcommand{\le}{\leqslant}\renewcommand{\leq}{\leqslant}
\renewcommand{\ge}{\geqslant}\renewcommand{\geq}{\geqslant}

\usepackage{tikz}
\usepackage[pdftex]{hyperref}
\usetikzlibrary{matrix,arrows}

\theoremstyle{definition}

\newcommand{\bbC}{\mathbb{C}}

\newcommand{\bbQ}{\mathbb{Q}}
\newcommand{\bbR}{\mathbb{R}}

\newcommand{\bbZ}{\mathbb{Z}}   

\newcommand{\cN}{\mathcal{N}}

\renewcommand{\and}{\quad \mbox{and} \quad}  
\renewcommand{\le}{\leqslant}\renewcommand{\leq}{\leqslant}
\renewcommand{\ge}{\geqslant}\renewcommand{\geq}{\geqslant}

\title{An extension of Deligne-Henniart's twisting formula and its applications}

\subjclass[2010]{11S37; 20C15, 11R39)\\Keywords: Local fields, Galois representations, local root numbers,
Converse theorems}

\author[Biswas]{\bfseries Sazzad Ali Biswas}

\address{
Department of Mathematical Sciences\\ 
University of Copenhagen\\ 
Universitetsparken 5, Copenhagen 2100\\
Denmark}

\address{SRM University-AP\\ 
Department of Mathematics\\
Neerukonda-Kuragallu Village Mangalagiri Mandal\\
Andhra Pradesh - 522502\\
India}
\email{sazzadali.b@srmap.edu.in, sazzad.jumath@gmail.com}

\begin{document}

\vspace{10mm}
\setcounter{page}{1}
\thispagestyle{empty}

\begin{abstract}

Let $F/\bbQ_p$ be a non-Archimedean local field, and $G_F$ be the absolute Galois group of $F$. 
Let $\rho_1$ and $\rho_2$ be two 
finite-dimensional complex representations of $G_F$. Let $\psi$ be a nontrivial additive character of $F$.
Then, the question is:
\begin{center}
 {\it What is the twisting formula for the root number $W(\rho_1\otimes\rho_2,\psi)$?}
\end{center}
In general, the answer to this question is not yet known. However, if one of $\rho_i  \quad(i=1,2)$ is one-dimensional 
with ``sufficiently''
large conductor,
then in \cite{D1}, Deligne gave a twisting formula for $W(\rho_1\otimes\rho_2,\psi)$. Later, in \cite{DH}, 
Deligne and Henniart gave a general twisting formula for a {\it zero}-dimensional {\it virtual} representation 
twisted by a finite-dimensional representation of $G_F$. 
In this paper, we first extend Deligne's twisting formula for U-isotropic
Heisenberg 
representation of dimension prime $p$,
then we further extend 
Deligne-Henniart's result.


Finally, we provide two very important applications of our twisting formula: -- \\ 
{\it (i) invariant formula for the local root numbers for U-isotropic Heisenberg representations, and (ii) 
a converse theorem on the Galois side.}

\end{abstract}

\maketitle

\section{{\bf Introduction}}

Let $F$ be a non-Archimedean local field of characteristic zero, i.e., a finite extension of $\bbQ_p$, where $p$ is a prime.
Let $G_F$ be  the absolute 
Galois representation of $F$. Let $\psi$ be a nontrivial additive character of $F$.
In this paper, we fix $\psi$ as a
nontrivial character of $F$.
For a given multiplication character
$\chi:F^\times\to \bbC^\times$
of $F$, we have an explicit formula for the root number $W(\chi,\psi)$ (cf. \cite{JT2}). 
If $\chi_1$ and $\chi_2$ are two unramified characters of $F^\times$, we have
 \begin{equation}
  W(\chi_1\chi_2,\psi)=W(\chi_1,\psi)W(\chi_2,\psi).
 \end{equation}
Further,  let $\chi_1$ be ramified and $\chi_2$ be unramified, then (cf. \cite{JT2}, (3.2.6.3))
\begin{equation}
 W(\chi_1\chi_2,\psi)=\chi_2(\pi_F)^{a(\chi_1)+n(\psi)}\cdot W(\chi_1,\psi).
\end{equation}
Here $a(\chi_1)$ (resp. $n(\psi)$) is the conductor of $\chi_1$ (resp. $\psi$), and $\pi_F$ is an uniformizer of $F$. 
We also have a twisting formula for the local root numbers by Deligne (cf. \cite{D1}, Lemma 4.16)
under some special conditions, which is as follows:\\
{\it Let $\chi_1$ and $\chi_2$ be two multiplicative characters of a local field $F$ such 
that $a(\chi_1)\geq 2\cdot a(\chi_2)$.
Let $y_{\chi_1,\psi}$ be an element of $F^\times$ such that 
$$\chi_1(1+x)=\psi(y_{\chi_1,\psi}x)$$
for all $x\in F$ with valuation $\nu_F(x)\geq\frac{a(\chi_1)}{2}$ (if $a(\chi_1)=0$, $y_{\chi_1,\psi}=\pi_{F}^{-n(\psi)}$). 
Then, 
\begin{equation}\label{eqn 2.3.17}
 W(\chi_1\chi_2,\psi)=\chi_2^{-1}(y_{\chi_1,\psi})\cdot W(\chi_1,\psi).
\end{equation}
}
For characters {\it without} any restriction, we also have a general twisting formula for characters 
(cf. Theorem 3.5 on p. 592 of \cite{SABTF}).

Moreover, if a finite-dimensional Galois representation $\rho$ is twisted by an unramified character
$\omega_s(x):=q_F^{-s\nu_F(x)}$,
we have the following twisting formula (cf. \cite{JT2} (3.4.5)):\\
\begin{equation}
 W(\rho\omega_{s},\psi)=W(\rho,\psi)\cdot \omega_{s}(c_{\rho,\psi})
\end{equation}
for any $c=c_{\rho,\psi}$ such that $\nu_F(c)=a(\rho)+n(\psi)\mathrm{dim}(\rho)$, 
where $a(\rho)$ is the Artin conductor (cf. Definition 2.1)  of 
the representation $\rho$.

Let $\rho_1$ and $\rho_2$ be two arbitrary finite-dimensional representations of $G_F$. Now the question is:
\begin{center}
{\it Is there any explicit formula for $W(\rho_1\otimes\rho_2,\psi)$?}
\end{center}
The answer to this question is {\it not} yet known. However, under some {\bf special conditions} -- 
when any of $\rho_i(i=1,2)$ 
is {\it one-dimensional} with {\it sufficiently large conductor} --
then Deligne gives an explicit formula for $W(\rho_1\otimes\rho_2,\psi)$
(cf. \cite{D1}, Subsection 4.1):\\
{\it Let $\rho_1=\rho$ be a finite-dimensional representation of $G_F$, and let $\rho_2=\chi$ be any nontrivial character
of $F^\times$. For each $\chi$ there exists an element $c\in F^\times$ such that
$$\chi(1+y)=\psi(cy)\quad\text{for sufficiently small $y$}.$$
For all $\chi$ with {\bf sufficiently large} conductors, we have the following formula:
\begin{equation}\label{eqn 1.5}
 W(\rho\otimes\chi,\psi)=W(\chi,\psi)^{\dim(\rho)}\cdot \det(\rho)(c^{-1}).
\end{equation}
}

{\bf However, for arbitrary characters (especially characters with smaller conductors), 
the equation (\ref{eqn 1.5}) is not true.} 
If $\rho$ is a minimal U-isotropic Heisenberg representation (for U-isotropic Heisenberg 
representations, see definition \ref{Definition U-isotropic}, 
and for minimal conductor Heisenberg representation, see Remark 4.4), then in this paper, we extend the above 
result (\ref{eqn 1.5}) of Deligne (cf. Theorem \ref{Theorem using Deligne-Henniart} below). 
 In Lemma \ref{Lemma U-equivalent}, we show that the U-isotropic Heisenberg representations are induced from 
some (tame) character of an unramified extension.

Moreover, by the construction (cf. \cite{D1}, \cite{RL}) of local root numbers, we can attach a local root number for 
any virtual representation of $G_F$.
Therefore, if we define a {\bf zero-dimensional} virtual representation from the representation $\rho$ as follows: 
$$\rho_0:=\rho-\dim(\rho)\cdot 1_{G_F},$$
where $1_{G_F}$ is the trivial representation
of $G_F$, then from Equation (\ref{eqn 1.5}), we have 
\begin{equation}\label{eqn 1.6}
    W(\rho_0\otimes\chi,\psi)=\det(\rho_0)(c^{-1}).
\end{equation}
Further, in \cite{DH}, Deligne and Henniart generalize the above result (\ref{eqn 1.6}) (see Section 4 of \cite{DH}),
in which $\chi$ is 
replaced by an arbitrary finite-dimensional representation $\rho$ of $G_F$, and the condition on the conductor of $\chi$ 
becomes a condition on the Artin conductor of $\rho$.

\begin{thm}[{\bf Deligne-Henniart's Twisting formula}, Theorem 4.6, \cite{DH}]\label{Deligne-Henniart's general formula}
 Let $\rho$ be a virtual representation of $G_F$ (without moderate component).
 There exists an element
 $\gamma\in F^\times$ uniquely determined modulo $U_F^{j(\rho)/2-1}$, such that for any virtual representation $\rho_0$ 
 of $G_F$ of dimension zero which satisfies $\beta(\rho_0)<j(\rho)/2$, we have 
 \begin{equation}
  W(\rho_0\otimes\rho,\psi)=\det(\rho_0)(\gamma).
 \end{equation}
Here $\nu_F(\gamma)=a(\rho)+\dim(\rho)\cdot n(\psi)$, and $j(\rho)$ is the jump of $\rho$. $\beta(\rho_0)$ is the 
maximum jump among all components of $\rho_0$.
\end{thm}

In this article, {\it we first generalize Deligne's twisting formula (\ref{eqn 1.5}) for a {\bf minimal} conductor
U-isotropic
Heisenberg 
representation $\rho_0$ of $G_F$ of dimension prime to $p$, and twisted by an {\bf arbitrary character $\chi_F$ with conductor
$a(\chi_F)\ge 2$}}. In the following twisting formula, the conductor of $\chi_F$ needs {\bf not be sufficiently large}, but 
Deligne's twisting formula (\ref{eqn 1.5}) is only true for those characters with {\it sufficiently large conductors}. 

 \begin{thm}\label{Theorem using Deligne-Henniart}
  Let $\rho_m=\rho(X_\eta,\chi_K)=\rho_0\otimes\widetilde{\chi_F}$ be a U-isotropic Heisenberg representation of the absolute Galois group
  $G_F$ of a non-archimedean local field $F/\bbQ_p$ of dimension $m$ with $gcd(m,p)=1$, where 
  $\rho_0=\rho_0(X_\eta,\chi_0)$ is a minimal conductor U-isotropic Heisenberg representation of $G_F$
  and $\widetilde{\chi_F}:G_F\to\bbC^\times$
  corresponds to $\chi_F:F^\times\to \bbC^\times$ by class field theory. 
 If $a(\chi_F)\ge 2$, then we have 
 \begin{equation}\label{eqn 5.4.9}
  W(\rho_m,\psi)=W(\rho_0\otimes\widetilde{\chi_F},\psi)=W(\chi_F,\psi)^m\cdot\det(\rho_0)(c),
 \end{equation}
where $c:=c(\chi_F,\psi)\in F^\times$ satisfies 
\begin{center}
 $\chi_F(1+x)=\psi(c^{-1}x)$ for all $x\in P_{F}^{a(\chi_F)-[\frac{a(\chi_F)}{2}]}$.
\end{center}
\end{thm}

{\bf Note:} When $a_F(\rho_0)=m a_0$ and if we assume that $a_0\ne 0$ and $a_F(\chi_F)\ge 2 a_0$, 
then in Proposition 1.2 of \cite{BZ}, we have an extension of the Deligne's twisting formula. However, in the above 
theorem, we only need the condition $a_F(\chi_F)\ge 2$. The arguments of the proof of Theorem 1.2 (cf. p. 19) are different
from 
the proof of Proposition 1.2 in \cite{BZ}.\\

Then, using Theorem \ref{Theorem using Deligne-Henniart}, we extend Deligne-Henniart's theorem
\ref{Deligne-Henniart's general formula}. In Theorem 1.1, the dimension of the virtual representation $\rho_0$ is zero.
However, in the following extension theorem, we have a
{\bf twisting formula for the representation $\sigma\otimes\rho_{m}$, where $\sigma$ is an arbitrary finite-dimensional
complex representation of $G_F$ and $\rho_m$ is a U-isotropic Heisenberg representation
of dimension $m$ prime to $p$.}

\begin{thm}\label{General twisting formula by minimal Heisenberg representation}
 Let $\sigma$ be an arbitrary finite dimension complex representation of $G_F$.
  If $a(\chi_F)\ge 2$ and $j(\rho_m)> 2\cdot \beta(\sigma)$, then we have 
  \begin{equation}
W(\sigma\otimes\rho_{m},\psi)= \det(\sigma)(\gamma)\cdot W(\chi_F,\psi)^{dim(\sigma\otimes\rho_{m})}
\cdot\det(\rho_{0})(c^{\dim(\sigma)}).   
  \end{equation}
Here  $\rho_m$ and $c\in F^\times$ are the same as in Theorem \ref{Theorem using Deligne-Henniart}, and 
$\nu_F(\gamma)=a(\rho_{m})+m\cdot n(\psi)$. 
\end{thm}

In Section 6, we also provide two very important applications of Theorem 
\ref{General twisting formula by minimal Heisenberg representation}:\\
{\it (i) Invariant formula of the root number for U-isotropic 
Heisenberg representations; and\\
(ii) a local converse theorem on the Galois side.}

From the dimension theorem (cf. Theorem \ref{Dimension Theorem}), we know that the dimension of a Heisenberg representation
$\rho$ of $G_F$ is of the form $dim(\rho)=p^r\cdot m$, where $r\ge 0$ and $gcd(m,p)=1$, $m|(q_F-1)$. 
And a U-isotropic Heisenberg representation $\rho$ can be expressed as $\rho=\rho_p\otimes\rho_m$, where
$\rho_p$ and $\rho_m$
are U-isotropic Heisenberg representations of $G_F$ of dimensions $p^r$ and $m$, respectively.
Therefore, to give an invariant
formula for the root number $W(\rho,\psi)$, we can use Theorem \ref{General twisting formula by minimal Heisenberg representation}, and 
we obtain the following result.

\begin{thm}[{\bf Invariant Formula}]\label{Invariant formula}
 Let $\rho$ be a U-isotropic Heisenberg representation of $G_F$ of the form $\rho=\rho_p\otimes\rho_{m}$ with 
 $\dim(\rho_p)=p^r (r\ge 1)$, and $\dim(\rho_{m})=m$, and $gcd(m,p)=1$. If jump: $j(\rho_m)>2\cdot j(\rho_p)$, we have 
 $$W(\rho,\psi)=W(\rho_p\otimes\rho_m,\psi)=\det(\rho_p)(\gamma)\cdot W(\chi_F,\psi)^{\dim(\rho)}\det(\rho_{0})(c^{p^r}).$$
 Here $\chi_F$, $c$, $\rho_0$ are same as in Theorem \ref{Theorem using Deligne-Henniart},
 and $\nu_F(\gamma)=a(\rho_{m})+m\cdot n(\psi)$.
\end{thm}

Using Theorem \ref{General twisting formula by minimal Heisenberg representation},
we obtain the following converse theorem on the 
Galois side.
\begin{thm}[{\bf Converse Theorem on the Galois side}]\label{Local Converse theorem}
 Let $\rho_{m}=\rho_0\otimes\widetilde{\chi_F}$ be a U-isotopic Heisenberg representation of
 $G_F$ of dimension prime to $p$.
 Let $\rho_1$, $\rho_2$ be two finite-dimensional complex representations of $G_F$ with
 $$\det(\rho_1)\equiv\det(\rho_2),\quad\text{and $j(\rho_m)>2\cdot \textrm{max}\{\beta(\rho_1),\beta(\rho_2)\}$}.$$
If 
 $$W(\rho_1\otimes\rho_{m},\psi)=W(\rho_2\otimes\rho_{m},\psi),$$
then $\rho_1\equiv\rho_2$ or $\rho_1\equiv \rho_2\otimes\mu$, where $\mu:F^\times\to\bbC^\times$ is an unramified character 
whose order divides $\dim(\rho_i), i=1,2$.
\end{thm}

In the Appendix, we provide an explicit description of the Heisenberg representations of dimension prime to $p$.

\section{{\bf Preliminaries and Notation}}

Let $F$  be a non-archimedean local field of characteristic zero, i.e., a finite extension
of the field $\mathbb{Q}_p$ (field of $p$-adic numbers),
where $p$ is a prime.
Let $K/F$ be a finite extension of the field $F$. Let $e_{K/F}$ be the ramification index for extension
$K/F$, and $f_{K/F}$ be 
the residue degree of extension $K/F$.


Let $O_F$ be the 
ring of integers in $F$, $P_F=\pi_F O_F$ the unique prime ideal in $O_F$, 
and $\pi_F$ a uniformizer, i.e., an element in $P_F$ whose valuation is one, i.e.,
 $\nu_F(\pi_F)=1$.
Let $U_F=O_F-P_F$ be the group of units in $O_F$.
Let $P_{F}^{i}=\{x\in F:\nu_F(x)\geq i\}$ and for $i\geq 0$ define $U_{F}^i=1+P_{F}^{i}$
(with proviso $U_{F}^{0}=U_F=O_{F}^{\times}$).
We also let $a(\chi)$ be the conductor of 
 nontrivial character $\chi: F^\times\to \mathbb{C}^\times$, i.e., $a(\chi)$ is the smallest integer $\geq 0$ such 
 that $\chi$ is trivial
 on $U_{F}^{a(\chi)}$. We say $\chi$ is unramified if the conductor of $\chi$ is {\bf zero} and otherwise ramified.
The {\bf conductor} of any nontrivial additive character $\psi$ of the field $F$ is an integer $n(\psi)$ if $\psi$ is trivial
on $P_{F}^{-n(\psi)}$, but nontrivial on $P_{F}^{-n(\psi)-1}$. 

Let $G_F:=Gal(\overline{F}/F)$ (resp. $W_F$) be the absolute Galois (resp. Weil group) of the field $F$, where 
$\overline{F}$ is an absolute algebraic closure of $F$.

\subsection{Ramification break}

Let $K/F$ be a Galois extension of $F$, and $G$ be the Galois group of the extension $K/F$.
For each $i\ge -1$, we define the 
$i$-th ramification subgroup of $G$ (in the lower numbering) as follows:
\begin{center}
 $G_i=\{\sigma\in G|\quad v_K(\sigma(\alpha)-\alpha)\geq i+1\quad \text{for all $\alpha\in O_K$}\}$.
 \end{center}
 An integer $t$ is called a {\bf ramification break or jump} for the extension $K/F$ or the ramification groups 
 $\{ G_i\}_{i\ge -1}$ if 
 $$G_t\ne G_{t+1}.$$
We also know that there is a decreasing filtration (with upper numbering) of $G$, and which is defined by the 
{\bf Hasse-Herbrand} function $\Psi=\Psi_{K/F}$ as follows:
$$G^u=G_{\Psi(u)}, \quad\text{where $u\in\bbR,\, u\ge -1$}.$$
By the definition of the Hasse-Herbrand function, $\Psi(-1)=-1, \Psi(0)=0$, we have
$G^{-1}=G_{-1}=G$, and $G^0=G_0$. Thus, a real number $t\ge -1$ is called a ramification break for $K/F$ or the filtration
$\{G^i\}_{i\ge -1}$ if 
$$G^t\ne G^{t+\varepsilon},\quad\text{for all $\varepsilon> 0$}.$$
When $G$ is {\bf abelian}, it can be proved (cf. {\bf Hasse-Arf theorem}, \cite{FV}, p. 91) that the ramification
breaks for $G$ are {\bf integers.} However, in general, the set of ramification breaks of a Galois group of a local field is 
{\it countably infinite, and need not consist of integers}.


\begin{dfn}[{\bf Artin and Swan conductors}]
 Let $G$ be a finite group and $R(G)$ be the complex representation ring of $G$. For any two representations 
 $\rho_1,\rho_2\in R(G)$ with characters $\chi_1,\chi_2$ respectively, we have Schur's inner product:
 $$<\rho_1,\rho_2>_G=<\chi_1,\chi_2>_G:=\frac{1}{|G|}\sum_{g\in G}\chi_1(g)\cdot\overline{\chi_2(g)}.$$
 Let $K/F$ be a finite Galois extension with Galois
 group $G:=\rm{Gal}(K/F)$. For an element $g\in G$ different from identity $1$, we define the non-negative integer 
 (cf. \cite{JPS}, Chapter IV, p. 62)
 $$i_G(g):=\rm{inf}\{\nu_K(x-g(x))|\; x\in O_K\}.$$
 Using this non-negative (when $g\ne 1$) integer $i_G(g)$, we define a function $a_G:G\to\bbZ$ as follows:
 \begin{center}
  $a_G(g)=-f_{K/F}\cdot i_G(g)$ when $g\ne 1$, and $a_G(1)=f_{K/F}\sum_{g\ne 1}i_G(g)$.
 \end{center}
Thus, from this definition, we can see that $\sum_{g\in G}a_G(g)=0$, hence $<a_G, 1_G>=0$. 
It can be proved (cf. \cite{JPS}, p. 99, Theorem 1) that the function $a_G$ is the character of a linear representation of $G$,
and that the corresponding linear representation is called the {\bf Artin representation} $A_G$ of $G$.

Similarly, for a nontrivial $g\ne 1\in G$, we define (cf. \cite{VS}, p. 247)
$$s_G(g)=\rm{inf}\{\nu_K(1-g(x)x^{-1})|\;x\in K^\times\},\qquad s_G(1)=-\sum_{g\ne 1}s_G(g).$$
We can define a function $\rm{sw}_G:G\to\bbZ$ as follows:
$$\rm{sw}_G(g)=-f_{K/F}\cdot s_G(g).$$
It can also be shown that $\rm{sw}_G$ is a character of a linear representation of $G$, and that the 
corresponding representation
is called the {\bf Swan representation} $SW_G$ of $G$.

From \cite{JP}, p. 160, we have the following relation between the Artin and Swan representations (cf. \cite{VS}, p. 248, equation (6.1.9))
\begin{equation}\label{eqn 5.1.22}
 SW_G=A_G+\rm{Ind}_{G_0}^{G}(1)-\rm{Ind}_{\{1\}}^{G}(1),
\end{equation}
$G_0$ is the $0$-th ramification group (i.e., inertia group) of $G$.

Now, we are in a position to define the Artin and Swan conductors of a representation $\rho\in R(G)$. The Artin conductor of a 
representation $\rho\in R(G)$ is defined as follows:
$$a_F(\rho):=<A_G,\rho>_G=<a_G,\chi>_G,$$
where $\chi$ is the character of 
the representation $\rho$. Similarly, for the representation $\rho$, the Swan conductor is defined as follows
$$\rm{sw}_F(\rho):=<SW_G,\rho>_G=<\rm{sw}_G,\chi>_G.$$
For more details about the Artin and Swan conductors, see Chapter 6 of \cite{VS} and Chapter VI of \cite{JPS}.
\end{dfn}
From Equation (\ref{eqn 5.1.22}), we obtain
\begin{equation}\label{eqn 5.1.23}
 a_F(\rho)=\rm{sw}_F(\rho)+\rm{dim}(\rho)-<1,\rho>_{G_0}.
\end{equation}
Moreover, from the Corollary of Proposition 4 on p. 101 of \cite{JPS}, for an induced representation 
$\rho:=\rm{Ind}_{\rm{Gal}(K/E)}^{\rm{Gal}(K/F)}(\rho_E)=\rm{Ind}_{E/F}(\rho_E)$, we have
\begin{equation}\label{eqn 5.1.24}
 a_F(\rho)=f_{E/F}\cdot \left( d_{E/F}\cdot \rm{dim}(\rho_E)+\textrm{a}_E(\rho_E)\right).
\end{equation}
We apply formula (\ref{eqn 5.1.24}) for $\rho_E=\chi_E$ of dimension $1$, and then conversely 
\begin{equation}
    a(\chi_E)=\frac{a_F(\rho)}{f_{E/F}}-d_{E/F},
\end{equation}
where $d_{E/F}$ is the exponent of the different of the extension $E/F$.
Therefore, if we know $a_F(\rho)$, we can compute the conductor $a(\chi_E)$ of $\chi_E$. 


\begin{dfn}[{\bf Jump for a representation}]
Let $\rho$ be an irreducible representation of $G$. For this irreducible $\rho$, we define jump for $\rho$ as follows:
$$j(\rho):=\rm{max}\{ i\;|\; \rho|_{G^i}\not\equiv 1\}.$$
Now, if $\rho$ is a ramified irreducible representation of $G$, 
then $\rho|_{I}\not\equiv 1$, where $I=G^0=G_0$ is the inertia subgroup
of $G$. Thus, from the definition of $j(\rho)$, we can say, if $\rho$ is irreducible, then we always have 
$j(\rho)\ge 0$, i.e., $\rho$ is nontrivial on the inertia group $G_0$. From the definitions of Swan and Artin 
conductors, and equation (\ref{eqn 5.1.23}), when $\rho$ is {\bf irreducible}, we have the following relations:
\begin{equation}\label{eqn 5.1.281}
 \rm{sw}_F(\rho)=\rm{dim}(\rho)\cdot j(\rho),\qquad a_F(\rho)=\rm{dim}(\rho)\cdot (j(\rho)+1).
\end{equation}
From the Theorem of Hasse-Arf (cf. \cite{JPS}, p. 76), if $\rm{dim}(\rho)=1$, i.e., $\rho$ is a character of $G/[G,G]$, 
we can say that $j(\rho)$ must be an integer, then $\rm{sw}_F(\rho)=j(\rho), a_F(\rho)=j(\rho)+1$.
Moreover, by class field theory, $\rho$ corresponds to a linear character $\chi_F$, hence for linear character $\chi_F$, we can write 
$$j(\chi_F):=\rm{max}\{i\;|\;\chi_F|_{U_F^i}\not\equiv1\},$$
because in class field theory (under Artin isomorphism), 
the upper numbering in the filtration of $\rm{Gal}(F_{ab}/F)$ is compatible with the filtration (descending chain) of the group of units 
$U_F$.


Similarly, jump can be defined for any {\bf virtual} representations of $G_F$.
Let $\rho$ be a virtual representation of $G_F$. We denote $j(\rho)$ (resp. $\beta(\rho)$) as the lower (resp. upper)
bound of $j(\rho_i)$, when $\rho_i$ runs over all components of $\rho.$ That is, if 
$$\rho=\sum_i^n n_i\rho_i,$$
then 
$$j(\rho)\le \{ j(\rho_1),\cdots,j(\rho_n)\}\le \beta(\rho).$$
Further, $j(\rho)\ge \alpha$ means that the components of $\rho$ do not have any non-null vector fixed by $G_F^\alpha$:
$\rho^{G_F^\alpha}=0$. And $\beta(\rho)<\beta$ (we then have $\beta>0$) means that $\rho$ comes by inflation from a 
virtual representation of $G_F/{G_F^\beta}$.

\end{dfn}

\subsection{Heisenberg Representations}

Let $G$ be a profinite group (in our case, $G=G_F$).
An irreducible representation $\rho$ of $G$ is called a \textbf{Heisenberg 
representation} if it represents commutators by 
scalar matrices. Therefore, higher commutators are represented by $1$ (see \cite{Z3}).
We can see that the linear characters of $G$ are Heisenberg representations as the degenerate special case.
To classify Heisenberg representations, we mention two invariants of an irreducible representation 
$\rho\in\rm{Irr}(G)$:
\begin{enumerate}
 \item Let $Z_\rho$ be the \textbf{scalar} group of $\rho$, i.e., $Z_\rho\subseteq G$ and $\rho(z)=\text{scalar matrix}$
for every $z\in Z_\rho$. If $V/\bbC$ is a representation space of $\rho$, we get $Z_\rho$ as the kernel of the composite map 
\begin{equation}\label{eqn 2.6.1}
 G\xrightarrow{\rho}GL_{\bbC}(V)\xrightarrow{\pi} PGL_{\bbC}(V)=GL_{\bbC}(V)/\bbC^\times E,
\end{equation}
where $E$ is the unit matrix, and denote $\overline{\rho}:=\pi\circ\rho$.
Therefore, $Z_\rho$ is a normal subgroup of $G$.
\item Let $\chi_\rho$ be the character of $Z_\rho$. This is given as $\rho(g)=\chi_\rho(g)\cdot E$ for all $g\in Z_\rho$. 
It can be seen that $\chi_\rho$ is a $G$-invariant character of $Z_\rho$, and this character, we call the central 
character of $\rho$.
\end{enumerate}
Let $A$ be a profinite abelian group. Then, we know that (cf. \cite{Z5}, p. 124, Theorem 1 and Theorem 2)
the set of isomorphism classes $\rm{PI}(A)$ of projective irreducible representations (for 
projective representation, see \cite{CR}, \S  51) of $A$ is in bijective correspondence with the 
set of continuous alternating characters $\rm{Alt}(A)$. If $\rho\in\rm{PI}(A)$ corresponds to $X\in\rm{Alt}(A)$, then 
\begin{center}
 $\rm{Ker}(\rho)=\rm{Rad}(X)$ \hspace{.4cm} and \hspace{.2cm}$[A:\rm{Rad}(X)]=\rm{dim}(\rho)^2$,
\end{center}
where $\rm{Rad}(X):=\{a\in A|\, X(a,b)=1,\,\text{for all}\, b\in A\}$, the {\bf radical of $X$}.

Let $A:=G/[G,G]$, so $A$ is abelian. 
We also know from the  composite map (\ref{eqn 2.6.1}),
$\overline{\rho}$ is a projective irreducible representation of $G$ and $Z_\rho$ is the kernel of $\overline{\rho}$.
Therefore, \textbf{modulo commutator group $[G,G]$}, we can consider that 
$\overline{\rho}$ is in $\rm{PI}(A)$, which corresponds to an alternating 
character $X$ of $A$ with kernel of $\overline{\rho}$ is $Z_\rho/[G,G]=\rm{Rad}(X)$.
We also know that 
$$[A:\rm{Rad}(X)]=[G/[G,G]:Z_\rho/[G,G]]=[G:Z_\rho].$$
Then, we observe that 
$$\rm{dim}(\overline{\rho})=\rm{dim}(\rho)=\sqrt{[G:Z_\rho]}.$$

Let $H$ be a subgroup of $A$. Then, we define the orthogonal complement of $H$ in $A$ with respect to $X$ as follows:
$$H^\perp:=\{a\in A:\quad X(a, H)\equiv1\}.$$
An {\bf isotropic} subgroup $H\subset A$ is a subgroup such that $H\subseteq H^\perp$ (cf. \cite{EWZ}, p. 270, Lemma 1(v)).
When an isotropic subgroup $H$ is maximal,
we call $H$ is a \textbf{maximal isotropic} for $X$. Thus, when $H$ is maximal isotropic, we have 
$H=H^\perp$.



Let
$C^1G=G$, $C^{i+1}G=[C^iG,G]$ denote the 
descending central series of $G$. Now assume that every projective
representation of $A$ lifts to an ordinary representation 
of $G$. Then, from I. Schur's results (cf. \cite{CR}, p. 361, Theorem 53.7), we have (cf. \cite{Z5}, p. 124, Theorem 2):
\begin{enumerate}
 \item Let $A\wedge_\bbZ A$ denote the alternating square of $\bbZ$-module $A$. The commutator map 
 \begin{equation}\label{eqn 2.6.3}
  A\wedge_\bbZ A\cong C^2G/C^3G, \hspace{.3cm} a\wedge b\mapsto [\hat{a},\hat{b}]
 \end{equation}
is an isomorphism.
\item The map $\rho\to X_\rho\in\rm{Alt}(A)$ from the Heisenberg representations to the alternating characters on $A$ is 
surjective. 
\end{enumerate}

\section{\textbf{U-isotropic Heisenberg representations}}

\subsection{{\bf Arithmetic Description of Heisenberg Representations}}

Let $F/\bbQ_p$ be a local field, and $\overline{F}$ be an algebraic closure of $F$.
Denote $G_F=\rm{Gal}(\overline{F}/F)$ as the 
absolute Galois group for $\overline{F}/F$. We know that (cf. \cite{HK2}, p. 197) each representation $\rho:G_F\to GL(n,\bbC)$ corresponds 
to a projective 
representation $\overline{\rho}:G_F\to GL(n,\bbC)\to PGL(n,\bbC)$. On the other hand, each projective representation 
$\overline{\rho}:G_F\to PGL(n,\bbC)$ can be lifted to a representation $\rho:G_F\to GL(n,\bbC)$.
Let $A_F=G_{F}^{ab}$ be the factor commutator group of $G_F$. Define 
\begin{center}
 $FF^\times:=\varprojlim(F^\times/N\wedge F^\times/N)$
\end{center}
where $N$ runs over all open subgroups of the finite index in $F^\times$. Denote by $\rm{Alt}(F^\times)$ as the set of 
all alternating characters $X:F^\times\times F^\times\to\bbC^\times$ such that $[F^\times:\rm{Rad}(X)]<\infty$. Then, the local 
reciprocity map gives an isomorphism between $A_F$ and the profinite completion of $F^\times$, and induces a natural bijection 
\begin{equation}
 \rm{PI}(A_F)\xrightarrow{\sim}\rm{Alt}(F^\times),
\end{equation}
where $\rm{PI}(A_F)$ is the set of isomorphism classes of projective irreducible representations of $A_F$.
Using class field theory, from the commutator map (\ref{eqn 2.6.3}) (cf. p. 125 of \cite{Z5}), we obtain 
\begin{equation}\label{eqn 5.1.2}
 c:FF^\times\cong [G_F,G_F]/[[G_F,G_F], G_F].
\end{equation}
 
Let $K/F$ be an abelian extension corresponding to the norm subgroup $N\subset F^\times$,
and if $W_{K/F}$ denotes the relative Weil 
group, the commutator map for $W_{K/F}$ induces the following isomorphism (cf. p. 128 of \cite{Z5})
\begin{equation}\label{eqn 5.1.3}
 c: F^\times/N\wedge F^\times/N\to K_{F}^{\times}/I_{F}K^\times,
\end{equation}
where 
\begin{center}
 $K_{F}^{\times}:=\{x\in K^\times|\quad N_{K/F}(x)=1\}$, i.e., the norm-1-subgroup of $K^\times$, and\\
 $I_FK^\times:=\{x^{1-\sigma}|\quad x\in K^{\times}, \sigma\in \rm{Gal}(K/F)\}<K_{F}^{\times}$, the augmentation with respect to $K/F$. 
\end{center}
Taking the projective limit over all abelian extensions $K/F$, the isomorphism (\ref{eqn 5.1.3}) induces:
\begin{equation}\label{eqn 5.1.4}
 c:FF^\times\cong \varprojlim K_{F}^{\times}/I_FK^\times,
\end{equation}
where the limit on the right side refers to the norm maps. 
This gives an {\bf arithmetic description} of the Heisenberg representation of 
group $G_F$.

\begin{thm}[Zink, \cite{Z2}, p. 301, Corollary 1.2]\label{Theorem 5.1.1}
 Set of Heisenberg representations $\rho$ of $G_F$ is in bijective correspondence
 with the set of all pairs $(X_\rho,\chi_\rho)$
 such that:
 \begin{enumerate}
  \item $X_\rho$ is a character of $FF^\times$,
  \item $\chi_\rho$ is a character of $K^{\times}/I_FK^\times$, where the abelian extension 
  $K/F$ corresponds to the radical 
  $N\subset F^\times$ of $X_\rho$, and 
  \item via (\ref{eqn 5.1.3}), the alternating character $X_\rho$ corresponds to the 
  restriction of $\chi_\rho$ to $K_{F}^{\times}$.
 \end{enumerate}

\end{thm}
\begin{rem} From the above theorem, we can say that for a 
 given pair $(X_\rho,\chi_\rho)$, we can construct the Heisenberg representation $\rho$ by induction
 from $G_K:=\rm{Gal}(\overline{F}/K)$ to 
 $G_F$:
 \begin{equation}\label{eqn 5.1.5}
 \sqrt{[F^\times:N]}\cdot\rho=\rm{Ind}_{K/F}(\chi_\rho),
 \end{equation}
where $N$ and $K$ are as in (2) of Theorem \ref{Theorem 5.1.1},
and the induction of $\chi_\rho$ (to be considered as a character of $G_K$ by class
field theory) produces a multiple of $\rho$. From 
$[F^\times:N]=[K:F]$, we obtain the {\bf dimension formula:}
\begin{equation}\label{eqn dimension formula}
 \rm{dim}(\rho)=\sqrt{[F^\times:N]},
\end{equation}
where $N$ is the radical of $X_\rho$.

\end{rem}

In the following lemma, we 
give the conditions of the existence of characters $\chi_E\in\widehat{E^\times}$ such that $\chi_E\circ N_{K/E}=\chi_K$, 
and the solution set 
of this $\chi_E$, where $K/E$ is a finite extension of $E$, and $\chi_K:K^\times\to\bbC^\times$ is a character of $K^\times$. 

\begin{lem}\label{Lemma 5.1.4}
Let $K/E$ be a finite extension of field $E$, and $\chi_K: K^\times\to\bbC^\times$.  
 \begin{enumerate}
  \item[(i)] The existence of characters $\chi_E: E^\times\to\bbC^\times$ such that $\chi_E\circ N_{K/E}=\chi_K$
  is equivalent to $K_{E}^{\times}\subset\rm{Ker}(\chi_K)$.
  \item[(ii)] If case (i) is fulfilled, we have a well-defined character 
  \begin{equation}
   \chi_{K/E}:=\chi_K\circ N_{K/E}^{-1}:\mathcal{N}_{K/E}\to \bbC^\times,
  \end{equation}
on the subgroup of norms $\mathcal{N}_{K/E}:=N_{K/E}(K^\times)\subset E^\times$, and the solutions $\chi_E$ such that 
$\chi_E\circ N_{K/E}=\chi_K$ are precisely the extensions of $\chi_{K/E}$ from $\mathcal{N}_{K/E}$ to a character of 
$E^\times$.
 \end{enumerate}
\end{lem}
\begin{proof}
{\bf (i)}
Suppose that an equation $\chi_K=\chi_E\circ N_{K/E}$ holds.
Let $x\in K_{E}^{\times}$, hence $N_{K/E}(x)=1$. Then, 
$$\chi_K(x)=\chi_E\circ N_{K/E}(x)=\chi_E(1)=1.$$
Therefore, $x\in\rm{Ker}(\chi_K)$, and hence $K_{E}^{\times}\subset \rm{Ker}(\chi_K)$.

Conversely, assume that $K_{E}^{\times}\subset\rm{Ker}(\chi_K)$. 
 Then, $\chi_K$ is actually a character of $K^\times/K_{E}^{\times}$. Again, we have
 $K^\times/K_{E}^{\times}\cong \mathcal{N}_{K/E}\subset E^\times$, 
 hence $\widehat{K^\times/K_{E}^{\times}}\cong \widehat{\mathcal{N}_{K/E}}$.
 Now, suppose that $\chi_K$ corresponds to the character $\chi_{K/E}$ of $\mathcal{N}_{K/E}$. Hence, 
 we can write $\chi_K\circ N_{K/F}^{-1}=\chi_{K/E}$. Thus, 
 the character $\chi_{K/E}:\mathcal{N}_{K/E}\to\bbC^\times$
 is well defined. Because $E^\times$ is an abelian group and $\mathcal{N}_{K/E}\subset E^\times$ is a subgroup of the 
 finite index
 (by class field theory) $[K:E]$,
 we can extend $\chi_{K/E}$ to $E^\times$, and $\chi_K$ is of the form $\chi_K=\chi_E\circ N_{K/E}$ 
 with $\chi_E|_{\cN_{K/E}}=\chi_{K/E}$.\\
 {\bf (ii)}
 If condition (i) is satisfied, then this part is obvious. 
 If $\chi_E$ is a solution of $\chi_K=\chi_E\circ N_{K/E}$, 
 with $\chi_{K/E}:=\chi_K\circ N_{K/E}^{-1}:\mathcal{N}_{K/E}\to\bbC^\times$, then certainly $\chi_E$ is an extension of 
 the character $\chi_{K/E}$. 
 
 Conversely, if $\chi_E$ extends $\chi_{K/E}$, then it is a solution of $\chi_K=\chi_E\circ N_{K/E}$ with 
 $\chi_K\circ N_{K/E}^{-1}=\chi_{K/E}:\mathcal{N}_{K/E}\to\bbC^\times$.
\end{proof}

\begin{rem}
Now take a Heisenberg representation $\rho=\rho(X,\chi_K)$ of $G_F$. Let $E/F$ be any extension corresponding to a maximal 
isotropic for $X$. In this Heisenberg setting, from Theorem \ref{Theorem 5.1.1}(2), we know $\chi_K$ is a character of 
$K^\times/I_FK^\times$, and from the first commutative diagram on p. 302 of \cite{Z2}, we have 
$N_{K/E}:K_F^\times/I_FK^\times\to E_F^\times/I_F\cN_{K/E}$. Thus, in the Heisenberg setting,
 we have more information than Lemma \ref{Lemma 5.1.4}(i), that $\chi_K$ is a character of 
 \begin{equation}
  K^\times/K_{E}^{\times}I_FK^\times\xrightarrow{N_{K/E}}\mathcal{N}_{K/E}/I_F\mathcal{N}_{K/E}\subset E^\times/I_F\mathcal{N}_{K/E},
 \end{equation}
and therefore, $\chi_{K/F}$ is actually a character of $\mathcal{N}_{K/E}/I_F\mathcal{N}_{K/E}$, or in other words, it is a 
$\rm{Gal}(E/F)$-invariant character of the $\rm{Gal}(E/F)$-module $\mathcal{N}_{K/E}\subset E^\times$. And if $\chi_E$ is one of 
the solutions of Lemma \ref{Lemma 5.1.4}(ii), then the complete solution is the set 
$\{\chi_E^\sigma\,|\,\sigma\in \rm{Gal}(E/F)\}$.

{\bf 
We know that $W(\chi_E,\psi\circ\rm{Tr}_{K/E})$ has the same value for all solutions $\chi_E$ of $\chi_E\circ N_{K/E}=\chi_K$,
which means for all $\chi_E$ which extend the character $\chi_{K/E}$}.

Moreover, from Lemma \ref{Lemma 5.1.4}, we can also see that $\chi_E|_{\mathcal{N}_{K/E}}=\chi_{K}\circ N_{K/E}^{-1}$.

\end{rem}

Let $\rho=\rho(X,\chi_K)$ be a Heisenberg representation of $G_F$. Let $E/F$ be any extension corresponding to a maximal 
isotropic for $X$. Then, using Lemma \ref{Lemma 5.1.4}, we have the following lemma. 

\begin{lem}\label{Lemma 5.1.44}
  Let $\rho=\rho(Z_\rho,\chi_\rho)=\rho(\rm{Gal}(L/K),\chi_K)$ be a Heisenberg representation of a finite local Galois group 
  $G=\rm{Gal}(L/F)$, where $F$ is a non-archimedean local field. Let $H=\rm{Gal}(L/E)$ be a maximal isotropic for 
  $\rho$. Then, we obtain
  \begin{equation}
   \rho=\rm{Ind}_{E/F}(\chi_{E}^{\sigma})\quad\text{for all $\sigma\in\rm{Gal}(E/F)$},
  \end{equation}
 where $\chi_E:E^\times/I_F\cN_{K/E}\to\bbC^\times$ with $\chi_K=\chi_E\circ N_{K/E}$.\\
 Moreover, for a fixed base field $E$ of a maximal isotropic for $\rho$, this construction of 
 $\rho$ is independent of the choice of this character $\chi_E$.
 \end{lem}

\begin{dfn}[{\bf U-isotropic}]\label{Definition U-isotropic}
Let $F$ be a non-archimedean local field. 
Let $X:FF^\times\to \bbC^\times$ be an alternating character with the property 
 $$X(\varepsilon_1,\varepsilon_2)=1,\qquad \text{for all $\varepsilon_1,\varepsilon_2\in U_F$}.$$
 In other words, $X$ is a character of $FF^\times/U_F\wedge U_F$. Then, $X$ is said to be {\bf  U-isotropic}.  
 These $X$ are easy to classify.
\end{dfn}

\begin{lem}[cf. Section 2.4 of \cite{BZ}]\label{Lemma U-isotropic}
 Fix a uniformizer $\pi_F$ and write $U:=U_F$. Then, we obtain an isomorphism 
 $$\widehat{U}\cong \widehat{FF^\times/U\wedge U}, \quad \eta\mapsto X_\eta,\quad \eta_X\leftarrow X$$
 between the characters of $U$ and U-isotropic alternating characters as follows:
 \begin{equation}\label{eqn 5.1.25}
  X_\eta(\pi_F^a\varepsilon_1,\pi_F^b\varepsilon_2):=\eta(\varepsilon_1)^b\cdot\eta(\varepsilon_2)^{-a},\quad
  \eta_X(\varepsilon):=X(\varepsilon,\pi_F),
 \end{equation}
 where $a,b\in\bbZ$, $\varepsilon,\varepsilon_1,\varepsilon_2\in U$, and $\eta:U\to\bbC^\times$.
 Then, 
 $$\rm{Rad}(X_\eta)=<\pi_F^{\#\eta}>\times\rm{Ker}(\eta)=<(\pi_F\varepsilon)^{\#\eta}>\times\rm{Ker}(\eta),$$
 does not depend on the choice of $\pi_F$, where  $\#\eta$ is the order of the character $\eta$, hence 
 $$F^\times/\rm{Rad}(X_\eta)\cong <\pi_F>/<\pi_F^{\#\eta}>\times U/\rm{Ker}(\eta)\cong \bbZ_{\#\eta}\times\bbZ_{\#\eta}.$$
 Therefore, all Heisenberg representations of type $\rho=\rho(X_\eta,\chi)$ have dimension $\rm{dim}(\rho)=\#\eta$.
\end{lem}

\begin{rem}
 
From Proposition 5.2(i) on p. 50 of \cite{Z4}, we know that 
$FF^\times/U_F\wedge U_F\cong U_F$, therefore, we have $\widehat{U_F}\cong \widehat{FF^\times/U_F\wedge U_F}$.
From Lemma \ref{Lemma U-isotropic}, we obtain
$$\rm{Rad}(X_\eta)=<\pi_F^{\#\eta}>\times\rm{Ker}(\eta).$$
Therefore, we can conclude that the
{\bf U-isotropic character $X=X_\eta$ has $U_F^i$ contained in its radical if and 
only if $\eta$ is a character of $U_F/U_{F}^{i}$.}

From Lemma \ref{Lemma U-isotropic}, we know that the dimension of a U-isotropic Heisenberg representation 
$\rho=\rho(X_\eta,\chi)$ of $G_F$ is $\rm{dim}(\rho)=\#\eta$, and $F^\times/\rm{Rad}(X_\eta)\cong \bbZ_{\#\eta}\times\bbZ_{\#\eta}$,
a direct product of two cyclic (bicyclic) groups of the same order $\#\eta$. 

\end{rem}

In the following lemma, we provide an equivalent condition for the {\bf U-isotropic Heisenberg representation}.

\begin{lem}\label{Lemma U-equivalent}
Let $G_F$ be the absolute Galois group of a non-archimedean local field $F$.
 For a Heisenberg representation $\rho=\rho(Z_\rho,\chi_\rho)=\rho(X,\chi_K)$, the following are equivalent:
 \begin{enumerate}
  \item The alternating character $X$ is U-isotropic.
  \item Let $E/F$ be the maximal unramified subextension in $K/F$. Then, $\rm{Gal}(K/E)$ is maximal isotropic for $X$.
  \item $\rho=\rm{Ind}_{E/F}(\chi_E)$ can be induced from a character $\chi_E$ of $E^\times$ (where $E$ is as in (2)).
 \end{enumerate}
\end{lem}
\begin{proof}
 {\bf $(1)\implies (2)$:}\\
 Here, the given condition, $X$ is U-isotropic. That is, $X\in\widehat{FF^\times/{U\wedge U}}$. We also know that
 $$\widehat{U}\cong \widehat{FF^\times/{U\wedge U}}.$$
This implies that $X$ corresponds to a character of $U$, namely $X\mapsto \eta_X$.
 Then, from Lemma \ref{Lemma U-isotropic}, we have
 $$F^\times /\rm{Rad}(X)\cong \bbZ_{m}\times\bbZ_{m},\quad\text{where $m=\#\eta_{X}$}.$$
 That is, $F^\times/\rm{Rad}(X)$ is the
 product of two cyclic groups of the same order.
 
 Because $K/F$ is the abelian bicyclic extension that corresponds to $\rm{Rad}(X)$, we can write:
 $$\cN_{K/F}=\rm{Rad}(X),\qquad\rm{Gal}(K/F)\cong F^\times/\rm{Rad}(X).$$
Let $E/F$ be the {\bf maximal unramified} subextension in $K/F$. Then, $[E:F]=m$ because the order of 
the maximal cyclic subgroup of $\rm{Gal}(K/F)$ is $m$. Then, $f_{E/F}=m$, hence 
$f_{K/F}=e_{K/F}=m$ because $f_{K/F}\cdot e_{K/F}=[K:F]=m^2$, and $\rm{Gal}(K/F)$ is {\bf not} a cyclic group.

Now, we must prove that the extension $E/F$ corresponds to a maximal isotropic for $X$.
Let $H/Z_\rho$ be a maximal isotropic for 
$X$, hence $[G_F/Z_\rho:H/Z_\rho]=m$, hence $H/Z_\rho=\rm{Gal}(K/E)$, i.e.,
the maximal unramified subextension $E/F$ in $K/F$ corresponds
to a maximal isotropic subgroup.\\

{\bf $(2)\implies (3):$}\\
Under the given condition, we have $E/F$ is the maximal unramified subextension in $K/F$, and $\rm{Gal}(K/E)$ is a maximal
isotropic for $X$.
Then, using Lemma \ref{Lemma 5.1.44}, we can write
\begin{center}
 $\rho(X,\chi_K)=\rm{Ind}_{E/F}(\chi_E)$,
 where $\chi_E:E^\times/I_F\mathcal{N}_{K/E}\to\bbC^\times$ with $\chi_E\circ N_{K/E}=\chi_K$.
\end{center}
{\bf $(3)\implies (1)$}:\\
Here, we have
\begin{center}
 $\rho(X,\chi_K)=\rm{Ind}_{E/F}(\chi_E)$, for $\chi_E\circ N_{K/E}=\chi_K$,
\end{center}
where $E/F$ is the maximum unramified subextension in $K/F$. 
Because $E/F$ is unramified, and the extension $E$ corresponds to a maximal isotropic subgroup for $X$, we have 
$U_F\subset\cN_{E/F}$, hence $U_F\subset\cN_{K/F}$, and $X|_{U\times U}=1$ because $U_F\subset F^\times\subset\cN_{K/E}$. 
This shows that $X$ is U-isotropic.
\end{proof}


\begin{cor}\label{Corollary U-isotropic}
 The U-isotropic Heisenberg representation $\rho=\rho(X_\eta,\chi)$ can never be wild.
Moreover, the representations $\rho$ of dimension prime to $p$ are precisely given as 
$\rho=\rho(X_\eta,\chi)$ for characters $\eta$ of $U/U^1.$
\end{cor}
\begin{proof}
In Lemma \ref{Lemma U-equivalent}, we see that the U-isotropic Heisenberg representation $\rho=\rho(X_\eta,\chi)$ is
induced from the character $\chi$ of the {\bf maximal unramified} subextension $E/F$ in $K/F$.
Therefore, $K/F$ containing $E/F$ cannot be wild extension. Hence, U-isotropic Heisenberg representations  are always tame.

From Lemma \ref{Lemma U-isotropic},
the dimension of $\rho=\rho(X_\eta,\chi)$ is 
the {\bf order} of the character $\eta:U/U^i\to\bbC^\times$ ($i\ge 1$). Because $\eta$ is a character of $U/U^i$,
the order of $\eta$ must be a divisor of $|U/U^i|=q_F^{i-1}(q_F-1).$

 We know that $|U/U^1|=q_F-1$ is prime to $p$. 
 We also know that the dimension $\rm{dim}(\rho)=\sqrt{[K:F]}=\sqrt{[F^\times:\rm{Rad}(X)]}$. 
 If $\rm{dim}(\rho)$ is prime to $p$, then 
 $K/F$ is tame, and $U_F^1\subseteq \rm{Rad}(X)$. However, $U/U^1$ is cyclic; hence $X$ is U-isotropic.
\end{proof}

\begin{rem}\label{Remark 5.1.3}
{\bf (1).}
Let $V_F$ be the wild ramification subgroup of $G_F$.
 We can show that $\rho|_{V_F}$ is irreducible if and only if $Z_\rho=G_K\subset G_F$
 corresponds to an abelian extension $K/F$ which is totally ramified, and wildly 
 ramified\footnote{Group theoretically, if $\rho|_{V_F}=\rm{Ind}_{H}^{G_F}(\chi_H)|_{V_F}$ is irreducible, then from 
 Section 7.4 of \cite{JP},
we can say $G_F=H\cdot V_F$. Here, $H=G_L$, where $L$ is a certain extension of $F$, and $V_F=G_{F_{mt}}$,
where $F_{mt}/F$ is the maximal 
tame extension of $F$. Therefore, $G_F=H\cdot V_F$ is equivalent to $F=L\cap F_{mt}$ that means
the extension $L/F$ must be totally 
ramified and wildly ramified, and $[G_F:H]=[L:F]=|V_F|$.
We know that the wild ramification subgroup $V_F$ is a pro-p-group (cf. \cite{FV}, p. 106). Then, 
 $\rm{dim}(\rho)$ is a power of $p$.} (cf. \cite{Z2}, p. 305). If $N:=N_{K/F}(K^\times)$ is the subgroup
 of norms, then this means that $N\cdot U_{F}^{1}=F^\times$, in other words,
 $$F^\times/N=N\cdot U_{F}^{1}/N=U_{F}^{1}/N\cap U_{F}^{1},$$
 where $N$ can also be considered as the radical of $X_\rho$. Therefore, we can consider the alternating
 character $X_\rho$ on the principal
 units $U_{F}^{1}\subset F^\times$. Then, 
 $$\rm{dim}(\rho)=\sqrt{[F^\times:N]}=\sqrt{[U_F^1: N\cap U_F^1]},$$
 is a power of $p$, because $U_F^1$ is a pro-p-group.

 Here we observe: If $\rho=\rho(X,\chi_K)$ with $\rho|_{V_F}$ stays irreducible, then
 $\rm{dim}(\rho)=p^n$, $n\ge 1$ and  
 $K/F$ is a totally and {\bf wildly} ramified. However, there is 
 a {\bf big} class of Heisenberg representations $\rho$ such that $\rm{dim}(\rho)=p^n$ is a $p$-power, but which are not 
 wild representations (see the Definition \ref{Definition U-isotropic} of U-isotropic).\\
 {\bf (2).}
Let $\rho=\rho(X,\chi_K)$ be a Heisenberg representation of the absolute Galois group $G_F$ of dimension $d>1$ which is prime 
to $p$. Then, from Lemma \ref{Lemma dimension equivalent}, we have  $d|q_F-1$. For this representation $\rho$, here 
$K/F$ must be tame if $\rm{Rad}(X)=\cN_{K/F}$ (cf. \cite{FV}, p. 115).
\end{rem}

\begin{lem}\label{Lemma dimension equivalent}
 Let $\rho=\rho(X,\chi_K)$ be a Heisenberg representation of 
 the absolute Galois group $G_F$ of a non-archimedean local field 
 $F/\bbQ_p$. Then, the following are equivalent:
 \begin{enumerate}
  \item $\rm{dim}(\rho)$ is prime to $p$.
  \item $\rm{dim}(\rho)$ is a divisor of $q_F-1$.
  \item The alternating character $X$ is U-isotropic and $X=X_\eta$ for a character $\eta$ of 
  $U_F/U_F^1$, i.e., $a(\eta)=1$.
  \item The abelian extension $K/F$, which corresponds to $\rm{Rad}(X)$ is tamely ramified.
 \end{enumerate}
\end{lem}
\begin{proof}
{\bf (1) implies (2):}\\
From Corollary \ref{Corollary U-isotropic}, we know that all Heisenberg representations of dimensions prime to $p$, are 
 U-isotropic representations of the form $\rho=\rho(X_\eta,\chi)$, where $\eta:U_F/U_F^1\to\bbC^\times$, and the dimensions 
 $\rm{dim}(\rho)=\#\eta$.
 Thus, if $\rm{dim}(\rho)$ is prime to $p$, then $\rm{dim}(\rho)=\#\eta$ is a divisor of $q_F-1$. \\
{\bf (2) implies (3)}:\\
 If $\rm{dim}(\rho)$
 is a divisor of $q_F-1$, then $gcd(p,\rm{dim}(\rho))=1$. From Corollary \ref{Corollary U-isotropic}, the alternating 
 character $X$ is U-isotropic and $X=X_\eta$ for a character $\eta\in\widehat{U_F/U_F^1}$.\\
 {\bf (3) implies (4)}:\\ 
 We know that 
 $$\rm{dim}(\rho)=\sqrt{[F^\times:\rm{Rad}(X_\eta)]}=\sqrt{[F^\times:\cN_{K/F}]}=\#\eta.$$
 Because $K/F$ is abelian, we have $\rm{dim}(\rho)^2=[K:F]$. Again, because $\#\eta=\rm{dim}(\rho)$ is a divisor
 of $q_F-1$, $K/F$ is tamely ramified.\\
 {\bf (4) implies (1):}\\
 If $K/F$ is tamely ramified, then we can write $U_F^1\subset \cN_{K/F}\subset F^\times$, and hence 
 $F^\times/\cN_{K/F}$ is a quotient group of $F^\times/U_F^1$. Therefore,  
if $K/F$ is the abelian tamely ramified extension, and $\cN_{K/F}=\rm{Rad}(X)$; then $X$ must be an alternating character
of $F^\times/U_F^1$.  
We also know that  
$F^\times=<\pi_F>\times<\zeta>\times U_{F}^{1}$, where $\zeta$ is a root of unity of order $q_F-1$. This implies 
$F^\times/U_{F}^{1}=<\pi_F>\times<\zeta>$.
Therefore, each element $x\in
F^\times/U_F^1$ can be written as $x= \pi_{F}^a\cdot \zeta^b$, where $a,b\in\bbZ$. 
We now take $x_1=\pi_{F}^{a_1}\zeta^{b_1}, x_2=\pi_{F}^{a_2}\zeta^{b_2}\in F^\times/U_{F}^{1}$, where $a_i,b_i\in\bbZ(i=1,2)$, then 
\begin{align*}
 X(x_1,x_2)
 &= X(\pi_{F}^{a_1}\zeta^{b_1},\; \pi_{F}^{a_2}\zeta^{b_2})\\
 &= X(\pi_{F}^{a_1},\zeta^{b_2})\cdot X(\zeta^{b_1},\pi_{F}^{a_2})\\
 &=\chi_\rho([\pi_{F}^{a_1},\zeta^{b_2}])\cdot\chi_\rho([\zeta^{b_1},\pi_{F}^{a_2}]).
\end{align*}
However, this implies  $X^{q_F-1}\equiv 1$ because $\zeta^{q_F-1}=1$,
which means that $X$ is actually an alternating character on  $F^\times/({F^\times}^{(q_F-1)} U_F^1),$ and therefore
$G_F/G_K$ is actually a quotient of $F^\times/({F^\times}^{(q_F-1)} U_F^1).$ 
We also know that $U_F^1$ is a pro-p-group and therefore 
$$U_F^1=(U_F^1)^{q_F-1}\subset F^\times.$$
Thus, the cardinality of 
$F^\times/({F^\times}^{(q_F-1)} U_F^1)$ is $(q_F-1)^2$ because 
$$F^\times/({F^\times}^{(q_F-1)} U_F^1)\cong \bbZ/(q_F-1)\bbZ\times<\zeta>\cong \bbZ_{q_F-1}\times\bbZ_{q_F-1}.$$
Therefore, $\rm{dim}(\rho)$ divides $q_F-1.$ Hence $\rm{dim}(\rho)$ is prime to $p$.

\end{proof}

\begin{rem}\label{Remark 5.1.14} 
Let $K_\eta|F$ be the abelian bicyclic 
extension that corresponds to $\rm{Rad}(X_\eta)$. Then, we can write

$$ \cN_{K_\eta/F}= \rm{Rad}(X_\eta),\qquad \rm{Gal}(K_\eta/F)\cong F^\times/\rm{Rad}(X_\eta),$$
and $f_{K_\eta|F}= e_{K_\eta|F}=\#\eta$. And the maximal unramified subextension 
$E/F\subset K_\eta/F$ corresponds to a maximal isotropic subgroup, hence
$$ \rho(X_\eta,\chi) = \rm{Ind}_{E/F}(\chi_E),\quad\textrm{for}\; \chi_E\circ N_{K_\eta/E} =\chi.$$
We recall here that $\chi:K_\eta^\times/I_FK_\eta^\times\rightarrow\bbC^\times$ is a character such that
(cf. Theorem \ref{Theorem 5.1.1}(3))
$$ \chi|_{(K_\eta^\times)_F} \leftrightarrow X_\eta,\quad\textrm{with respect to}\; 
(K_\eta^\times)_F/I_FK_\eta^\times\cong F^\times/\rm{Rad}(X_\eta)\wedge F^\times/\rm{Rad}(X_\eta).$$
In particular, we see that $(K_\eta^\times)_F/I_FK_\eta^\times$ is cyclic of 
order $\#\eta$, and $\chi|_{(K_\eta^\times)_F}$ must be a faithful character of that cyclic group.
\end{rem}

\section{{\bf Computation of the Artin and Swan conductors, and the dimension theorem}}

The Artin and Swan conductors are the main ingredients of the Deligne-Henniart's twisting 
formula (cf. Theorem \ref{Deligne-Henniart's general formula}), and Deligne's formula (\ref{eqn 1.5}). Therefore, for 
U-isotropic Heisenberg representations, we need to compute them explicitly. In this section, we perform all the necessary 
computations.

In the following theorem, we provide a general dimension formula for a Heisenberg representation.
\begin{thm}[{\bf Dimension}]\label{Dimension Theorem}
Let $F/\bbQ_p$ be a local field, and $G_F$ be the absolute Galois group of $F$. If $\rho$ is a Heisenberg representation of 
$G_F$, then $\rm{dim}(\rho)=p^n\cdot d'$, where $n (\ge 0)$ is an integer, and the prime to $p$ factor $d'$ must divide $q_F-1$.
\end{thm}
\begin{proof}
 By the definition of Heisenberg representation $\rho$, we have the following relation 
 $$[[G_F,G_F],G_F]\subseteq\rm{Ker}(\rho).$$
 Then, we can consider $\rho$ as a representation of $G:=G_F/[[G_F,G_F],G_F]$. Because 
 $[x,g]\in [[G_F,G_F],G_F]$ for all $x\in [G_F,G_F]$, and $g\in G_F$, we have $[G,G]=[G_F,G_F]/[[G_F,G_F],G_F]\subseteq Z(G)$,
 hence, $G$ is a two-step nilpotent group.
 
 We know that each nilpotent group is isomorphic to the direct product of its Sylow subgroups. Therefore, we can write 
 $$G=G_p\times G_{p'},$$
 where $G_p$ is the Sylow $p$-subgroup, and $G_{p'}$ is the direct product of all other Sylow subgroups.
 Therefore, each irreducible
 representation $\rho$ has the form $\rho=\rho_{p}\otimes\rho_{p'}$, where $\rho_{p}$ and $\rho_{p'}$ are irreducible representations of 
 $G_p$ and $G_{p'}$, respectively. 
 
 We also know that finite $p$-groups are nilpotent groups, and the direct product of a finite number of 
 nilpotent groups is again a nilpotent group.
 Therefore, $G_p$ and $G_{p'}$ are both two-step nilpotent groups because $G$ is a two-step nilpotent group.
 Therefore, the representations
 $\rho_p$ and $\rho_{p'}$ are both Heisenberg representations of $G_p$ and $G_{p'}$, respectively.
 
 To prove our assertion, we must show that $\rm{dim}(\rho_p)$ can be an arbitrary power of $p$, whereas 
 $\rm{dim}(\rho_{p'})$ must divide $q_F-1$. Because
 $\rho_p$ is an {\bf irreducible} representation of $p$-group $G_p$, the dimension of $\rho_p$ is some $p$-power.
 
 Again, from the construction of $\rho_{p'}$, we can say that $\rm{dim}(\rho_{p'})$ is {\bf prime} to $p$. 
 From Lemma \ref{Lemma dimension equivalent}, $\rm{dim}(\rho_{p'})$ is a divisor of $q_F-1$.

This completes the proof.

\end{proof}

Using Equation (\ref{eqn 5.1.24}), in our Heisenberg setting, we have the following proposition.

\begin{prop}\label{Proposition 5.1.20}
 Let $\rho=\rho(Z,\chi_\rho)=\rho(X,\chi_K)$ be a Heisenberg representation of the absolute Galois group $G_F$ of a field 
 $F/\bbQ_p$ of dimension $m$. Let $E/F$ be any subextension in $K/F$ corresponding to a maximal isotropic subgroup for $X$. 
 Then, 
 $$a_F(\rho)=a_F(\rm{Ind}_{E/F}(\chi_E)),\qquad m\cdot a_F(\rho)=a_F(\rm{Ind}_{K/F}(\chi_K)).$$
 As a consequence, we have 
 $$a(\chi_K)=e_{K/E}\cdot a(\chi_E)-d_{K/E}.$$
 In particular, $a(\chi_K)=a(\chi_E)$ if $K/E$ is unramified.
\end{prop}
\begin{proof}
 We know that $\rho=\rm{Ind}_{E/F}(\chi_E)$ and $m \cdot \rho=\rm{Ind}_{K/F}(\chi_K)$.
 By the definition of the Artin conductor, we can write 
 $$a_F(\rm{dim}(\rho)\cdot \rho)=\rm{dim}(\rho)\cdot a_F(\rho)=m\cdot a_F(\rm{Ind}_{E/F}(\chi_E)).$$
 Because $K/E/F$ is a tower of Galois extensions with $[K:E]=m=e_{K/E}f_{K/E}$, we have the following 
 transitivity relation of 
 different (cf. \cite{JPS}, p. 51,
 Proposition 8)
 $$\mathcal{D}_{K/F}=\mathcal{D}_{K/E}\cdot \mathcal{D}_{E/F}.$$
 From the definition of different of a Galois extension, and taking $K$-valuation we obtain
 \begin{equation}\label{eqn discriminant relation}
  d_{K/F}=d_{K/E}+e_{K/E}\cdot d_{E/F}.
 \end{equation}
 Now, using Equation (\ref{eqn 5.1.24}), we have
 \begin{equation}\label{eqn 44}
  m\cdot a_F(\rm{Ind}_{E/F}(\chi_E))=m\cdot f_{E/F}\left(d_{E/F}+a(\chi_E)\right)=m\cdot f_{E/F}\cdot d_{E/F}+e_{K/E}\cdot f_{K/F}
  \cdot a(\chi_E),
 \end{equation}
and 
\begin{equation}\label{eqn 45}
 a_F(\rm{Ind}_{K/F}(\chi_K))=f_{K/F}\cdot\left(d_{K/F}+a(\chi_K)\right)=f_{K/F}\cdot d_{K/F}+f_{K/F}\cdot a(\chi_K).
\end{equation}
Using Equation (\ref{eqn discriminant relation}), from Equations (\ref{eqn 44}), (\ref{eqn 45}), we have 
$$a(\chi_K)=e_{K/E}\cdot a(\chi_E)-d_{K/E}.$$
When $K/E$ is unramified, i.e., $e_{K/E}=1$ and $d_{K/E}=0$, hence $a(\chi_K)=a(\chi_E)$.

\end{proof}

\begin{dfn}[{\bf Jump for alternating character}]

For each $X\in\widehat{FF^\times}$, we define jump as follows:
\begin{equation}
 j(X):=\begin{cases}
        0 & \text{when $X$ is trivial}\\
        \rm{max}\{i\;|\; X|_{UU^i}\not\equiv 1\} & \text{when $X$ is nontrivial},
       \end{cases}
\end{equation}
where $UU^i\subseteq FF^\times$ is a subgroup that under (\ref{eqn 5.1.2}) corresponds 
$$G_F^i\cap[G_F,G_F]/G_F^i\cap[[G_F,G_F],G_F]\subseteq[G_F,G_F]/[[G_F,G_F],G_F].$$
\end{dfn}

\begin{rem}

Let $\rho=\rho(X_\rho,\chi_K)$ be a {\bf minimal conductor} (i.e., a representation with the smallest Artin conductor) 
Heisenberg representation for $X_\rho$ of the absolute Galois group $G_F$. 
From Theorem 3 on p. 125 of \cite{Z5}, we 
have 
\begin{equation}\label{eqn 5.1.26}
 \rm{sw}_F(\rho)=\rm{dim}(\rho)\cdot j(X_\rho)=\sqrt{[F^\times:\rm{Rad}(X_\rho)]}\cdot j(X_\rho).
\end{equation}
Moreover, if $\rho_0=\rho_0(X,\chi_0)$ is a minimal representation
corresponding $X$, then all other Heisenberg
representations of dimension $\rm{dim}(\rho)$ are of the form $\rho=\chi_F\otimes \rho_0=(X, (\chi_F\circ N_{K/F})\chi_0)$,
where $\chi_F:F^\times\to \bbC^\times$. Then, 
we have (cf. \cite{Z2}, p. 305, equation (5))
\begin{equation}\label{eqn 5.1.27}
 \rm{sw}_F(\rho)=\rm{sw}_F(\chi_F\otimes\rho_0)=\sqrt{[F^\times:\rm{Rad}(X)]}\cdot\rm{max}\{j(\chi_F), j(X)\}.
\end{equation}

\end{rem}

For a minimal conductor U-isotopic Heisenberg representation, we have the following proposition.

\begin{prop}\label{Proposition conductor}
 Let $\rho=\rho(X_\eta,\chi_K)$ be a U-isotropic Heisenberg representation of $G_F$ of the minimal conductor. 
 Then, we have the following conductor relation
 \begin{center}
  $j(X_\eta)=j(\eta)$, $\rm{sw}_F(\rho)=\rm{dim}(\rho)\cdot j(X_\eta)=\#\eta\cdot j(\eta)$,
  $a_F(\rho)=\rm{sw}_F(\rho)+\rm{dim}(\rho)=\#\eta(j(\eta)+1)=\#\eta\cdot a_F(\eta)$.
 \end{center}
 
\end{prop}

\begin{proof}
 From \cite{Z5}, on p. 126, Proposition 4(i) and Proposition 5(ii), and $U\wedge U=U^1\wedge U^1$, we see that the injection 
$U^i\wedge F^\times\subseteq UU^i$ induces a natural isomorphism 
$$U^i\wedge<\pi_F>\cong UU^{i}/UU^i\cap (U\wedge U)$$
for all $i\ge 0$. 

Now, let $j(X_\eta)=n-1$, hence $X_\eta|_{UU^n}=1$ but $X_\eta|_{UU^{n-1}}\ne 1$.
This gives $X_\eta|_{U^n\wedge<\pi_F>}=1$ but $X_\eta|_{U^{n-1}\wedge<\pi_F>}\ne 1$. From Equation (\ref{eqn 5.1.25}),
we can conclude that $\eta(x)=1$ for all $x\in U^n$ but $\eta(x)\ne 1$ for $x\in U^{n-1}$. Hence 
$$j(\eta)=n-1=j(X_\eta).$$
From the definition of $j(\chi)$, where $\chi$ is a character of $F^\times$, we can see that 
$j(\chi)=a(\chi)-1$, i.e., $a(\chi)=j(\chi)+1$.

From Equation (\ref{eqn 5.1.26}), we obtain:
$$\rm{sw}_F(\rho)=\rm{dim}(\rho)\cdot j(X_\eta)=\#\eta\cdot j(\eta),$$
since $\rm{dim}(\rho)=\#\eta$ and $j(X_\eta)=j(\eta)$. Finally, from  equation (\ref{eqn 5.1.23}) for $\rho$ (here $<1,\rho>_{G_0}=0$),
we have 
\begin{equation}\label{eqn 5.1.28}
 a_F(\rho)=\rm{sw}_F(\rho)+\rm{dim}(\rho)=\#\eta\cdot j(\eta)+\#\eta=\#\eta\cdot (j(\eta)+1)=\#\eta\cdot a_F(\eta).
\end{equation}
\end{proof}

From Proposition \ref{Proposition 5.1.20} and Proposition \ref{Proposition conductor}, 
we obtain the following result.

\begin{lem}\label{Lemma general conductor}
 Let $\rho=\rho(X_\eta,\chi_K)$ be a U-isotopic Heisenberg representation 
 of the minimal conductor of the absolute Galois group $G_F$ of a non-archimedean
 local field $F$.
 Let $K=K_\eta$ correspond to the radical of $X_\eta$, and let $E_1/F$ be the maximal unramified subextension, and $E/F$
 be any maximal cyclic and totally ramified subextension in $K/F$. Let $m$ denote the order of $\eta$.
 Then, $\rho$ is induced by $\chi_{E_1}$ or by 
 $\chi_E$ respectively, and we have 
 \begin{enumerate}
  \item $a_E(\chi_E)=m\cdot a(\eta)-d_{E/F}$,
  \item $a_{E_1}(\chi_{E_1})=a(\eta)$,
  \item and for the character $\chi_K\in\widehat{K^\times}$,
  $$a_K(\chi_K)=m\cdot a(\eta)-d_{K/F}.$$
 \end{enumerate}
Moreover, $a_E(\chi_E)=a_K(\chi_K)$. 
\end{lem}
\begin{proof}
The proof of these assertions follows from Equation (\ref{eqn 5.1.24}) and Proposition \ref{Proposition conductor}. When 
$\rho=\rm{Ind}_{E/F}(\chi_E)$, where $E/F$ is a maximal cyclic and totally ramified subextension in $K/F$, from Equation 
(\ref{eqn 5.1.24}), we obtain
\begin{align*}
 a_F(\rho)
 &=m\cdot a(\eta)\quad\text{using Proposition $\ref{Proposition conductor}$},\\
 &=f_{E/F}\cdot\left(d_{E/F}\cdot 1+a_E(\chi_E)\right),\quad\text{since $\rho=\rm{Ind}_{E/F}(\chi_E)$}\\
 &=1\cdot\left(d_{E/F}+a_E(\chi_E)\right),
\end{align*}
because $E/F$ is totally ramified; hence $f_{E/F}=1$.  This implies $a_E(\chi_E)=m\cdot a(\eta)-d_{E/F}$.

Similarly, when $\rho=\rm{Ind}_{E_1/F}(\chi_{E_1})$, where $E_1/F$ is the maximal unramified subextension in $K/F$; hence 
$f_{E_1/F}=m$ and $d_{E_1/F}=0$, using Equation (\ref{eqn 5.1.24}), we obtain $a_{E_1}(\chi_{E_1})= a(\eta)$.

Again, from Proposition \ref{Proposition 5.1.20} we have 
$$a_K(\chi_K)=m\cdot a(\chi_{E_1})-d_{K/E_1}=m\cdot a(\eta)-d_{K/F}.$$

Finally, because $E/F$ is a maximal cyclic totally ramified implies that $K/E$ is unramified, and therefore 
$$d_{E/F}=d_{K/F},\quad\text{and hence}\; a_E(\chi_E)=a_K(\chi_K).$$
\end{proof}

\begin{rem}\label{Remark 5.1.22}
 Assume that we are in the dimension $m=\#\eta$ prime to $p$ case. Then, from Corollary \ref{Corollary U-isotropic}, $\eta$
must be a character of $U/U^1$ (for $U=U_F$), hence
$$ a(\eta)=1\qquad  a_F(\rho_0) =m.$$
Therefore, in this case, the minimal conductor of $\rho$ is $m$, hence it is equal to the dimension of $\rho$. 

From Lemma \ref{Lemma general conductor}, in this case, we have 
$$a_{E_1}(\chi_{E_1})=a(\eta)=1.$$
And $K/F, E/F$ are tamely ramified of the ramification exponent $e_{K/F}=m$, hence
$$ a_E(\chi_E) = a_K(\chi_K) = m\cdot a(\eta)-d_{K/F}=m -(e_{K/F}-1)=m-(m-1)=1.$$
Thus, we can conclude that in this case, all three characters (i.e., $\chi_{E_1},\chi_E$, and $\chi_K$) are of conductor $1$.

In the general case $a_{E_1}(\chi_{E_1}) = a(\eta)$ and
$$a_E(\chi_E)= a_K(\chi_K) = m\cdot a(\eta)-d,$$
where $d=d_{E/F}=d_{K/F}$, the conductors will be different.
\end{rem}

In general, if $\rho=\rho_0\otimes\chi_F$, where $\rho_0$ is a finite-dimensional 
minimal conductor representation of $G_F$, and 
$\chi_F\in\widehat{F^\times}$, then we obtain the following result.

\begin{lem}\label{Lemma 5.1.23}
 Let $\rho_0$ be a finite-dimensional representation of $G_F$ of the minimal conductor.
 Then, we have 
 \begin{equation}
  a_F(\rho)=\rm{dim}(\rho_0)\cdot a_F(\chi_F),
 \end{equation}
where $\rho=\rho_0\otimes\chi_F=\rho(X_\eta,(\chi_F\circ N_{K/F})\chi_0)$ and $\chi_F\in\widehat{F^\times}$ with 
$a(\chi_F)>\frac{a(\rho_0)}{\rm{dim}(\rho)}$.
\end{lem}
\begin{proof}
From Equation (\ref{eqn 5.1.281}), we have $a_F(\rho_0)=\rm{dim}(\rho_0)\cdot (1+j(\rho_0))$.
Under the given condition, $\rho_0$ is a minimal conductor representation. 
Therefore, for representation $\rho=\rho_0\otimes\chi_F$, we have 
\begin{align*}
 a_F(\rho)
 &=a_F(\rho_0\otimes\chi_F)=\rm{dim}(\rho_0)\cdot\left(1+\rm{max}\{j(\rho_0),j(\chi_F)\}\right)\\
 &=\rm{dim}(\rho_0)\cdot\rm{max}\{1+j(\chi_F), 1+j(\rho_0)\}\\
 &=\rm{dim}(\rho_0)\cdot\rm{max}\{a(\chi_F), 1+j(\rho_0)\}\\
 &=\rm{dim}(\rho_0)\cdot a_F(\chi_F),
\end{align*}
because by the given condition 
$$a(\chi_F)>\frac{a(\rho_0)}{\rm{dim}(\rho_0)}=\frac{\rm{dim}(\rho_0)\cdot(1+j(\rho_0))}{\rm{dim}(\rho_0)}=1+j(\rho_0).$$

\end{proof}

\begin{prop}\label{Proposition 5.1.23}
 Let $\rho=\rho(X,\chi_K)$ be a Heisenberg representation of dimension $m$ of the absolute Galois group $G_F$ of a 
 non-archimedean local field $F$.
 Then, $m| a_F(\rho)$ if and only if:\\
$X$ is U-isotropic, or (if $X$ is not U-isotropic) $a_F(\rho)$ is with respect to $X$ not the minimal conductor.
\end{prop}
\begin{proof}
From Lemma \ref{Lemma 5.1.23}, we know that if $\rho$ is not minimal, $a_F(\rho)$ is always a multiple of the 
dimension $m$. Therefore, now we only have to check for minimal conductors. 
In the U-isotropic case, the minimal conductor is multiple
of the dimension (cf. Proposition \ref{Proposition conductor}). 

Finally, suppose that $X$ is not U-isotropic, i.e., $X|_{U\wedge U}=X|_{U^1\wedge U^1}\not\equiv1$, because 
$U\wedge U=U^1\wedge U^1$ (see the Remark on p. 126 of \cite{Z5}). We also know that 
$UU^i=(UU^i\cap U^1\wedge U^1)\times(U^i\wedge<\pi_F>)$ (cf. \cite{Z5}, p. 126, Proposition 5(ii)). 
In Proposition 5 of \cite{Z5}, we observe that all the jumps $v$ in the filtration $\{UU^i\cap (U^1\wedge U^1)\}, i\in\bbR_{+}$
are not {\bf integers with $v>1$}. This shows that $j(X)$ is also not an integer, hence $a_F(\rho_0)$ is not a 
multiple of the dimension. This implies that the conductor $a_F(\rho)$ is not minimal.

\end{proof}

For a minimal conductor Heisenberg representation, we have the following theorem.
\begin{prop}\label{Proposition-A}
 Let $\rho=\rho(X_\eta,\chi_K)$ be a Heisenberg representation of the 
 absolute Galois group $G_F$ of a non-archimedean local field $F/\bbQ_p$ of
 dimension $m$ prime to $p$. Then, it is of minimal conductor $a_F(\rho)=m$ if and only if $\rho$ is a representation of 
 $G_F/V_F$, where $V_F$ is the subgroup of wild ramification.
\end{prop}

\begin{proof}
Under the given condition, the dimension $\rm{dim}(\rho)=m$ is prime to $p$. 
Then, from Lemma \ref{Lemma dimension equivalent}
we can conclude that $K/F$ is tamely ramified with $f_{K/F}=e_{K/F}=m$ 
(cf. Remark \ref{Remark 5.1.14}), and hence $d_{K/F}=e_{K/F}-1=m-1$. Then, from the conductor 
formula (\ref{eqn 5.1.24}), we can easily see that {\bf $a(\rho)=m$ is minimal if and only if $a(\chi_K)=1$}.

Further, for some extension $L/K$, if $\cN_{L/K}=\rm{Ker}(\chi_K)$, then by class field theory, we can conclude:
{\bf $L/K$ is tamely ramified if and only if $a(\chi_K)=1$.}

Now suppose that $\rho$ is a Heisenberg representation of $G:=G_F/V_F$ of dimension $m$ prime to $p$. 
This implies $V_F\subset\rm{Ker}(\rho)=\rm{Ker}(\chi_K)=\cN_{L/K},$
where $L/K$ is some tamely ramified extension. Then, $a(\chi_K)=1$, hence $a(\rho)=m$ is minimal.

Conversely, when conductor $a(\rho)=m$ is minimal, we have $a(\chi_K)=1$. 
 By class field theory, this character $\chi_K$ determines
 an extension $L/K$ such that $\cN_{L/K}=\rm{Ker}(\chi_K)$. Becasue $a(\chi_K)=1$, here $L/K$ must be tamely ramified, hence 
 $L/F$ is tamely ramified.
 This means that $V_F$ sits in the kernel $\rm{Ker}(\rho)=G_L$ , therefore $\rho$ is actually a representation of 
 $G_F/V_F$.

\end{proof}

\section{{\bf Proof of Theorems \ref{Theorem using Deligne-Henniart} and 
\ref{General twisting formula by minimal Heisenberg representation}}}


\begin{proof}[{\bf Proof of Theorem 1.2}]
 {\bf Step-1:}
 Under the given conditions, 
 $$\rho_m=\rho_0\otimes\widetilde{\chi\chi_F},$$
 where $\rho_0$ is a minimal conductor U-isotropic Heisenberg representation of $G_F$ of dimension $m$ prime to $p$, and 
 $\widetilde{\chi_F}:G_F\to\bbC^\times$ corresponds to $\chi_F:F^\times\to\bbC^\times$ by class field theory. And
 \footnote{We also know that there are $m^2$ characters of $F^\times/{F^\times}^m$ such that 
$\rho_0\otimes\widetilde{\chi}=\rho_0$ (cf. \cite{Z2}, p. 303, Proposition 1.4). Therefore, we always have:
\begin{equation*}
 \rho=\rho_0\otimes\widetilde{\chi_F}=\rho_0\otimes\widetilde{\chi\chi_F},
\end{equation*}
where $\chi\in\widehat{F^\times/{F^\times}^m}$, and
$\widetilde{\chi_F}:G_F\to\bbC^\times$ corresponds to $\chi_F$ by class field theory.
}
 here $\chi:F^\times/{F^\times}^m\to\bbC^\times$ such that $\rho_0=\rho_0\otimes\widetilde{\chi}$.

 Let $\zeta$ be a $(q_F-1)$-st
root of unity. Because $U_F^1$ is a pro-p-group, and $gcd(p,m)=1$, we have 
\begin{equation}\label{eqn 5.2.23}
  F^\times/{F^\times}^m=<\pi_F>\times<\zeta>\times U_F^1/<\pi_F^m>\times<\zeta>^m\times U_F^1\cong \bbZ_m\times\bbZ_m,
\end{equation}
that is, a direct product of two cyclic groups of the same order. 
Hence $F^\times/{F^\times}^m\cong\widehat{F^\times/{F^\times}^m}$.
Because ${F^\times}^m=<\pi_F^m>\times<\zeta>^m\times U_F^1$, and
$F^\times/{F^\times}^m\cong \bbZ_m\times\bbZ_m,$
we have $a(\chi)\le 1$, and $\#\chi$ is a divisor of $m$ for all 
$\chi\in\widehat{F^\times/{F^\times}^m}$. Now, if we take a character $\chi_F$ of $F^\times$ of conductor $\ge 2$, hence  
$a(\chi_F)\ge 2 a(\chi)$ for all $\chi\in \widehat{F^\times/{F^\times}^m}$. Then, using Deligne's 
formula (\ref{eqn 2.3.17}), we have 
\begin{equation}\label{eqn 5.2.216}
 W(\chi\chi_F,\psi)^m=\chi(c)^m\cdot W(\chi_F,\psi)^m=W(\chi_F,\psi)^m,
\end{equation}
where $c\in F^\times$ with $\nu_F(c)=a(\chi_F)+n(\psi)$, satisfies 
$$\chi_F(1+x)=\psi(c^{-1}x),\quad\text{for all $x\in F^\times$ with $2\nu_F(x)\ge a(\chi)$}.$$

 From Proposition \ref{Proposition-A}, we can consider the representation $\rho_0$ as a representation
 of $G:=\rm{Gal}(F_{mt}/F)$, where $F_{mt}/F$ is the maximal tamely ramified subextension in $\bar{F}/F$.
 Then, we can write 
 $$\rho_0=\rm{Ind}_{E/F}(\chi_{E,0}),\quad and\quad \rho=\rm{Ind}_{E/F}(\chi_E),$$
 where $E/F$ is a cyclic tamely ramified subextension of $K/F$ of degree $m$, and 
 $$\chi_E:=\chi_{E,0}\otimes(\chi_F\circ N_{E/F}).$$
 
 {\bf Step-2:}
 Let $G$ be a finite group. Let $R(G)$ be the character ring provided with the tensor product as multiplication, and the 
 unit representation as a unit element. Then, for any zero-dimensional representations
 $\pi\in R(G)$, we have (cf. Theorem 2.1(h) on p. 40
 of \cite{RB}):
 \begin{equation}\label{eqn 5.2.217}
  \pi=\sum_{H\leq G}n_H\rm{Ind}_{H}^{G}(\chi_H-1_H),
 \end{equation}
 where $n_H\in \bbZ$ (cf. Proposition 2.24 on p. 48 of \cite{RB}) and $\chi_H\in\widehat{H}$. Moreover, from 
 Theorem 2.1 (k) of \cite{RB}, we know that $n_H\ne 0$ {\bf only if} $Z(G)\le H$ and $\chi_H|_{Z(G)}=\chi_Z$, where $Z(G)$ is 
 the center of $G$, and $\chi_Z$ is the center character.
 
 Now, we use the above formula (\ref{eqn 5.2.217}) for the representation $\rho_0-m\cdot 1_F$, and we obtain
 \begin{equation}\label{eqn 5.2.211}
 \rho_0 - m\cdot 1_F=\sum_{i=1}^{r}n_i\rm{Ind}_{E_i/F}(\chi_{E_i} - 1_{E_i}),
\end{equation}
where $E_i/F$ are intermediate fields of $K/F$, and for nonzero $n_i$, we obtain the following relation 
$$\chi_0=\chi_{E_i}\circ N_{K/E_i}.$$
Because $a(\chi_0)=1$, and $K/E_i$ are cyclic tamely ramified;\footnote{The subfields $E_i\subseteq K$ are related to Boltje's 
approach by $\rm{Gal}(K/E_i)=H_i/Z(G)$ and the fact that $\chi_{H_i}$ extends the character $\chi_Z$ which translates 
via class field theory to $\chi_{E_i}\circ N_{K/E_i}=\chi_K$. Moreover, $X=\chi_Z\circ [-,-]$ and $\chi_Z$ extendible to 
$H_i$ means that $\rm{Gal}(K/E_i)=H_i/Z(G)$ must be isotropic for $X$, hence in our situation 
$K/E_i$ must be a cyclic extension of degree dividing $m$.}, we have $a(\chi_{E_i})=1$ for all $i=1,2,\cdots,r$.
Then, from Equation (\ref{eqn 5.2.211}), we have
\begin{equation}\label{eqn 5.2.212}
 \det(\rho_0)=\prod_{i=1}^{r}(\chi_{E_i}|_{F^\times})^{n_i}.
\end{equation}

{\bf Step-3:} From Equation (\ref{eqn 5.2.211}), we can also write
\begin{equation}\label{eqn 5.2.213}
 \rho_0\otimes\widetilde{\chi\chi_F} - m\cdot\chi\chi_F=\sum_{i=1}^{r}n_i\rm{Ind}_{E_i/F}(\chi_{E_i}\theta_i-\theta_i),
\end{equation}
where $\theta_i:=\chi\chi_F\circ N_{E_i/F}$ for all $i\in\{1,2,\cdots,r\}$.
Because $a(\chi_F)\ge 2$, and $E_i/F$ are tamely ramified, 
the conductors $a(\theta_i)\ge 2$ for all $i\in\{1,2,\cdots,r\}$. Then, from Equation (\ref{eqn 5.2.213}), we can write
\begin{align}
 W(\rho,\psi)\nonumber
 &=W(\chi\chi_F,\psi)^m\cdot\prod_{i=1}^{r}W(\chi_{E_i}\theta_i-\theta_i,\psi_{E_i})^{n_i}\\\nonumber
 &=W(\chi\chi_F,\psi)^m\cdot\prod_{i=1}^{r}\frac{W(\chi_{E_i}\theta_i,\psi_{E_i})^{n_i}}{W(\theta_i,\psi_{E_i})^{n_i}}\\
 &=W(\chi\chi_F,\psi)^m\cdot\prod_{i=1}^{r}\chi_{E_i}(c_i)^{n_i},\label{eqn 5.2.214}
\end{align}
where $\psi_{E_i}=\psi\circ\rm{Tr}_{E_i/F}$, and $c_i\in E_{i}^{\times}$ such that 
\begin{center}
 $\theta_i(1+y)=\psi_{E_i}(\frac{y}{c_i}),$ for all $y\in P_{E_i}^{a(\theta_i)-[\frac{a(\theta_i)}{2}]}$.
\end{center}
Moreover, here $E_i/F$ are tamely ramified extensions, then from Lemma 18.1 of \cite{BH} on p. 123, we have
$$N_{E_i/F}(1+y)\cong 1+\rm{Tr}_{E_i/F}(y)\pmod{P_F^{a(\chi_F)}},$$
and 
$\rm{Tr}_{E_i/F}(y)\in P_F^{a(\chi_F)-[\frac{a(\chi_F)}{2}]}$. Therefore, for all 
$y\in P_{E_i}^{a(\theta_i)-[\frac{a(\theta_i)}{2}]}$, we can write (cf. Proposition 18.1 on p. 124 of \cite{BH}):
\begin{align*}
 \theta_i(1+y)
 &=\chi\chi_F\circ N_{E_i/F}(1+y)=\chi_F(1+\rm{Tr}_{E_i/F}(y))\\
 &=\psi(\frac{\rm{Tr}_{E_i/F}(y)}{c})=\psi_{E_i}(\frac{y}{c}),
\end{align*}
where $c:=c(\chi_F,\psi)$ for which $\chi_F(1+x)=\psi(\frac{x}{c})$ for $x\in P_F^{a(\chi_F)-[\frac{a(\chi_F)}{2}]}$.
This varifies that the choice $c_i(\theta_i,\psi_{E_i})=c_i(\chi\chi_F,\psi_{E_i})=c(\chi_F,\psi)\in F^{\times}$ is right 
for applying Tate-Lamprecht formula (cf. \cite{SABLT}).

Then, using Equation (\ref{eqn 5.2.212}) in Equation (\ref{eqn 5.2.214}), we obtain
\begin{equation}\label{eqn 5.2.215}
 W(\rho,\psi)=W(\chi\chi_F,\psi)^{m}\cdot \det(\rho_0)(c).
\end{equation}
Finally, using Equation (\ref{eqn 5.2.216}), from Equation (\ref{eqn 5.2.215}), we can write 
\begin{align*}
 W(\rho,\psi)
 &=W(\chi\chi_F,\psi)^{m}\cdot\det(\rho_0)(c)\\
 &=W(\chi_F,\psi)^m\cdot\det(\rho_0)(c).
\end{align*}
 
\end{proof}

Now we are in a position to prove Theorem \ref{General twisting formula by minimal Heisenberg representation}.


\begin{proof}[{\bf Proof of Theorem 1.3}]
The proof of the above assertion is the combination of Theorem \ref{Theorem using Deligne-Henniart}, and
Deligne-Henniart Theorem \ref{Deligne-Henniart's general formula}.

From the representation $\sigma$, we define a zero-dimensional virtual representation as follows:
$$\sigma_0:=\sigma-\dim(\sigma)\cdot 1_{G_F}.$$
Then,
\begin{equation}
  \det(\sigma_0)=\det(\sigma-\dim(\sigma)\cdot 1_{G_F})=\det(\sigma).
\end{equation}
Now, we want to use Theorem \ref{Deligne-Henniart's general formula} for the representation 
$\sigma_0$. 
It can be seen that 
$$\beta(\sigma_0)=\beta(\sigma-\dim(\sigma)\cdot 1_{G_F})=\beta(\sigma).$$
Hence, $\sigma_0$
satisfies the condition: $j(\rho_m)>2 \cdot \beta(\sigma)=2\cdot\beta(\sigma_0)$. Therefore, we can use 
Theorem \ref{Deligne-Henniart's general formula} for the representation $\sigma_0$, hence we can write
\begin{align*}
 W(\sigma_0\otimes\rho_{m},\psi)=\det(\sigma_0)(\gamma)\\
 =W((\sigma-\dim(\sigma)\cdot 1_{G_F})\otimes\rho_{m},\psi)\\
 =W(\sigma\otimes\rho_{m},\psi)\cdot W(\rho_{m},\psi)^{-\dim(\sigma)}.
\end{align*}
This gives 
\begin{equation}\label{eqn 5.11}
   W(\sigma\otimes\rho_{m},\psi)
 =\det(\sigma_0)(\gamma)\cdot W(\rho_{m},\psi)^{\dim(\sigma)}=\det(\sigma)(\gamma)\cdot W(\rho_m,\psi)^{\dim(\sigma)}.
\end{equation}
Now, we use Theorem \ref{Theorem using Deligne-Henniart} in equation (\ref{eqn 5.11}), and obtain
\begin{align*} 
W(\sigma\otimes\rho_{m},\psi)=
\det(\sigma)(\gamma)\cdot W(\chi_F,\psi)^{\dim(\sigma\otimes\rho_m)}\cdot\det(\rho_{0})(c^{\dim(\sigma)}).
\end{align*}

\end{proof}

\begin{rem}
 Suppose that $\rho$ is any arbitrary finite-dimensional Galois representation;
 $W(\rho,\psi)$ and $W(\rho\otimes\chi,\psi)$
 are explicitly known. Then, by the above method, under some conditions on jump of $\sigma$, 
 one can give an explicit twisting formulas for 
 $W(\sigma\otimes\rho,\psi)$ and $W(\sigma\otimes\rho',\psi)$, where $\rho':=\rho\otimes\chi$ as the above proof.
 

\end{rem}

\section{{\bf Applications}}

\subsection{{\bf Invariant root number formula for Heisenberg representations}}

From the construction of a Heisenberg representation $\rho=\rho(X,\chi_K)$ of $G_F$, we can write 
$$\rho=Ind_{E/F}(\chi_E),$$
where $K/E/F$ is the fixed field of a maximal isotropic subgroup $H=Gal(\overline{F}/E)$ of $\rho$.
The root number for $\rho$ is 
\begin{equation}\label{eqn 6.1}
 W(\rho,\psi)=\lambda_{E/F}(\psi)\cdot W(\chi_E,\psi).
\end{equation}
Here, $\lambda_{E/F}(\psi):=W(Ind_{E/F}(1_E),\psi)$ is the Langlands $\lambda$-function for the extension $E/F$ 
(cf. \cite{SABJNT}, \cite{SABJAA}).
Because for a given Heisenberg representation $\rho$, the maximal isotropic subgroups for $\rho$ are not {\bf unique}, and 
{\it there will be many maximal isotropic subgroups $H$ for $\rho$}, hence their fixed fields $E$ are also not unique.

Suppose that for a Heisenberg
representation $\rho$, we have two different maximal isotropic subgroups $H_1$, and $H_2$ of $G$. 
Let $E_1$ and $E_2$ be the 
fixed fields of $H_1$ and $H_2$, respectively. Then, we can write
$$(1)\quad \rho=Ind_{E_1/F}(\chi_{E_1}),\quad (2)\quad \rho=Ind_{E_2/F}(\chi_{E_2}).$$
Then, 
\begin{equation}\label{eqn 6.2}
 W(\rho,\psi)=\lambda_{E_1/F}(\psi)\cdot W(\chi_{E_1},\psi)=\lambda_{E_2/F}(\psi)\cdot W(\chi_{E_2},\psi).
\end{equation}
Now, if we notice equation (\ref{eqn 6.2}), the right-hand side depends
on $E_i$ ($i=1,2$). However, the root number for $\rho$ is unique
for its equivalence classes. Therefore, to give an explicit formula for $W(\rho,\psi)$, one needs to give 
an invariant formula, that is, it is independent of the choices of $E_i$. 
For a Heisenberg representation $\rho$ of $G_F$ dimension prime to $p$, one can see an explicit invariant formula for 
$W(\rho,\psi)$ in \cite{BZ}.

\begin{exm}
Let $\rho=\rho(X,\chi_K)$ be a two-dimensional Heisenberg representation of $G_F$, where $F/\bbQ_p$ and $p\ne 2$
(see Example (\ref{Example for Heisenberg reps}) in the Appendix for an explicit description of a two-dimensional Heisenberg
representations). Then, we have $Gal(K/F)\cong \bbZ/2\bbZ\times\bbZ/2\bbZ$. Therefore, there are the three maximal
isotropic 
subgroups $H_i (i=1,2,3)$ for $\rho$, and hence $E_i/F (i=1,2,3)$ are the three quadratic extensions of $F$. This can be proved 
(cf. \cite{SABJNT}, Lemma 4.7 on pp. 191-192) that $\lambda_{E_i/F}(\psi)$ are not the same. Similarly,
it can also be proved that
$W(\chi_{E_i},\psi)$ are also not the same. Therefore, we cannot use equation (\ref{eqn 6.2}) to give an explicit
formula for $W(\rho,\psi)$.

\end{exm}

Let $\rho$ be a U-isotropic Heisenberg of $G_F$. Then, the alternating
character $X=X_\rho$  corresponds to a character $\eta:U_F\to\bbC^\times$
of the group of units. Then, we have the unique decomposition of $\eta$ as follows:
$$\eta=\eta_p\cdot\eta',$$ 
where $\eta'$ is of order  prime to $p$ and the order of $\eta_p$
is a power of $p$.  Correspondingly:
$$X=X_p\cdot X', \quad\text{where $\eta_p\leftrightarrow X_p,\, \eta'\leftrightarrow X'$}.$$
For the representation $\rho$ this means:
\begin{equation}\label{eqn x}
 \rho=\rho_p\otimes\rho'
\end{equation}
where  $\dim(\rho_p)=p^r (r\ge 0),\quad and \quad \dim(\rho')=:m'$, and $gcd(m', p)=1$, $m'|(q_F-1)$.



Moreover, expression (\ref{eqn x}) is {\bf not} unique because 
$$\rho_p\otimes\rho'=\rho_p\omega^{-1}\otimes\omega\rho'$$
for any character $\omega$ of $F^\times$.  Thus for appropriate $\omega$ we may assume
that $a_F(\rho')=m'$ is {\bf minimal} and $a_F(\rho_p)=m_p a$,  where $a\ge a_F(\eta_p)$.
Then, we can use Theorem \ref{General twisting formula by minimal Heisenberg representation} to give 
an invariant formula for $W(\rho,\psi)=W(\rho_p\otimes\rho')$ when $j(\rho')>2\cdot j(\rho_p)$.


\begin{thm}[Invariant Formula]\label{Invariant formula}
 Let $\rho$ be a U-isotropic Heisenberg representation of $G_F$ of the form $\rho=\rho_p\otimes\rho_{m}$ with 
 $\dim(\rho_p)=p^r (r\ge 1)$, and $\dim(\rho_{m})=m$, and $gcd(m,p)=1$. If jump: $j(\rho_m)>2\cdot j(\rho_p)$, we have 
 $$W(\rho,\psi)=W(\rho_p\otimes\rho_m,\psi)=\det(\rho_p)(\gamma)\cdot W(\chi_F,\psi)^{\dim(\rho)}\det(\rho_{0})(c^{p^r}).$$
 Here $\chi_F$, $c$, $\rho_0$ are same as in Theorem \ref{Theorem using Deligne-Henniart},
 and $\nu_F(\gamma)=a(\rho_{m})+m\cdot n(\psi)$.
\end{thm}

\begin{proof}
 The idea behind the proof is the same as Theorem \ref{General twisting formula by minimal Heisenberg representation}.
 Here just replace $\sigma$ with $\rho_p$. 
 And another important thing is that the representation $\rho_p$ is Heisenberg, and hence 
 irreducible. Therefore, $\beta(\rho_p)=j(\rho_p)$.
\end{proof}

\begin{rem}
Because maximal isotropic subgroups for $\rho$ are not unique, giving an invariant formula for $\det(\rho)$, we cannot
simply use Gallagher's result (cf. Theorem 30.1.6 of \cite{GK}):
\begin{equation}
\det(\rho)(g)=\det(Ind_{H}^{G}(\chi_H))(g)=\Delta_{H}^{G}(g)\cdot \chi_H(T_{G/H}(g)),\quad \text{for all $g\in G$},
\end{equation}
where $\Delta_{H}^{G}$ is the determinant of $Ind_{H}^{G}(1_H)$, and $T_{G/H}$ is the transfer map from $G$ to $H$.

For the invariant formula of $\det(\rho)$, we can see Theorem 5.1 and Theorem 5.1.A of \cite{SABIJM}.

\end{rem}

\subsection{{\bf Converse theorem on the Galois side}}

We know the 
answer to the following question:\\
{\it How to construct a modular form from a given Dirichlet series with {\bf desirabale } properties (e.g., analytic continuation,
moderate growth, functional equation), i.e., starting with the series 
$$L(s)=\sum_{n=1}^{\infty}\frac{a_n}{n^s},$$
under what conditions is the function 
$$f(z)=\sum_{n=1}^{\infty}a_n e^{2\pi i nz}$$
a modular form for some Fuchsian group?}\\
The answer to this question is known as the {\it classical converse theorem} in number
theory (cf. \cite{HH}, \cite{EH}, \cite{AW}).
The classical converse theorems establish 
a one-to-one correspondence between ``nice'' Dirichlet series and automorphic functions. 
Traditionally, converse theorems 
have provided a way to characterize Dirichlet series associated with modular forms in terms of their analytic properties.

The modern version of the classical converse theorems is stated in terms of automorphic representations instead of modular
forms. Again, the Langlands local correspondence that automorphic representations are associated with 
Galois representations. Therefore, one can ask the following questions:\\
{\it (a). Are there any converse theorems for automorphic representations (automorphic side of the converse theorem)?\\
(b). Similarly, is there any converse theorem for Galois representations (Galois side of the converse theorem)?}

The answer to (a) is {\bf YES}. For local converse theorems on the automorphic side, refer to \cite{DJ}.
Let $G$ be a reductive group over a p-adic local field $F/\bbQ_p$. Let 
$^LG$ be the Langlands dual group of $G$, which is a {\bf semi-product} of the complex dual group $G^{\vee}$, and the absolute 
Galois group $G_F:=\rm{Gal}(\overline{F}/F)$.
Let $\phi: W_F\times SL_2(\bbC)\to {^LG}$
be a continuous homomorphism, and which is {\bf admissible}. The $G^\vee$-conjugacy class of 
such a homomorphism $\phi$ is called a {\bf local Langlands parameter}.
Let $\Phi(G/F)$ be the set of local Langlands parameters, and let $\Pi(G/F)$ be the set of equivalence 
classes of irreducible admissible representations of $G(F)$.
 
The local Langlands conjecture (cf. \cite{MH1}, \cite{MH2}, \cite{HT01}, \cite{H00}, \cite{PS}) for $G$ over $F$ asserts that for each local Langlands parameter $\phi\in\Phi(G/F)$,
there should be a {\bf finite} subset $\Pi(\phi)$, which is called the {\bf local $L$-packet}
attached to $\phi$ such that the 
set $\{\Pi(\phi)|\phi\in \Phi(G/F)\}$ is a partition of $\Pi(G/F)$, among other required properties. 
{\it The map $\phi\mapsto\Pi(\phi)$ is called the local Langlands correspondence or local Langlands {\bf reciprocity law} for
$G$ over $F$.}


\begin{rem}[{\bf $\gamma$-factors}]
We define the {\bf local $\gamma$-factors} as follows (cf. \cite{JPSS83}): 
\begin{equation}
 \gamma(s,\pi_1\times\pi_2,\psi):=W(s,\pi_1\times\pi_2,\psi)\cdot 
 \frac{L(1-s,\pi_1^V\times\pi_2^V)}{L(s,\pi_1\times\pi_2)}.
\end{equation}
On the $W_F\times SL_2(\bbC)$ side, one defines the $\gamma$-factor in the same way (cf. \cite{JT2}). For more information
about local factors, refer to \cite{DJ}, \cite{FG99}, \cite{GPSR97}.


\end{rem}

{\bf (a) Converse theorem on the automorphic side:}\\
Roughly, the local converse theorem is to find the smallest subcollection of twisted local 
$\gamma$-factors $\gamma(s,\pi\times\tau,\psi)$ which classifies the irreducible admissible representations $\pi$ 
up to isomorphism. However, this is usually not the case in general.
From the local Langlands conjecture, one may expect a certain subcollection of local $\gamma$-factors classifies the irreducible
representation $\pi$ up to $L$-packet. On the other hand, if the irreducible admissible representations under consideration 
have additional structures, then one may still expect a certain subcollection of local $\gamma$-factors classifies the irreducible
representation $\pi$ up to equivalence.

For $GL_n(F)$, we have the following theorem.

\begin{thm}[{\bf Jacquet-Liu, 2016, \cite{JL}}]
 Let $\pi_1,\pi_2$ be irreducible generic representations of $GL_n(F)$. Suppose that they have the same central character.
 If 
 $$\gamma(s,\pi_1\times\tau,\psi)=\gamma(s,\pi_2\times\tau,\psi)$$
 as functions of the complex variable $s$, for all irreducible generic representations $\tau$ of $GL_r(F)$ with 
 $1\le r\le[\frac{n}{2}]$, then $\pi_1\cong \pi_2$.
\end{thm}
\begin{rem}
 The above Jacquet-Liu's theorem was first conjectured by Jacquet (cf. Conjecture 1.1 of \cite{JNS}).
 On p. 170 of \cite{BH}, one can also see that the converse theorem 
 for $GL_2(F)$ is in terms of $L$- and $\epsilon$-factors. For further details on local converse theorems, 
 see \cite{GH}, \cite{CHEN1}, \cite{CPS2}.
\end{rem}


{\bf (b) Converse theorem on the Galois side:}\\
On the Galois side, we do not have any converse theorem similar to $GL_n$-side except
Volker Heiermann's \cite{VH}
work, because {\it we do not have a general twisting formula for the local root numbers for Galois representations.} 
Therefore,
to translate the local converse theorems from the automorphic side to the Galois side, we need a general twisting formula which
is not yet known.

However, using Theorem \ref{General twisting formula by minimal Heisenberg representation}, we have the following 
converse theorem on the Galois side for Heisenberg representations.

\begin{thm}[{\bf Converse Theorem on the Galois side}]\label{Local Converse theorem}
 Let $\rho_{m}=\rho_0\otimes\widetilde{\chi_F}$ be a U-isotropic Heisenberg representation of
 $G_F$ of dimension prime to $p$.
 Let $\rho_1$, $\rho_2$ be two finite-dimensional complex representations of $G_F$ with
 $$\det(\rho_1)\equiv\det(\rho_2),\quad\text{and $j(\rho_m)>2\cdot \textrm{max}\{\beta(\rho_1),\beta(\rho_2)\}$}.$$
If 
 $$W(\rho_1\otimes\rho_{m},\psi)=W(\rho_2\otimes\rho_{m},\psi),$$
then $\rho_1\equiv\rho_2$ or $\rho_1\equiv \rho_2\otimes\mu$, where $\mu:F^\times\to\bbC^\times$ is an unramified character 
whose order divides $\dim(\rho_i), i=1,2$.
\end{thm}

\begin{proof}
Under the given condition $j(\rho_m)>2\cdot\rm{max}\{\beta(\rho_1),\beta(\rho_2)\}$, then  
 using Theorem \ref{General twisting formula by minimal Heisenberg representation}, we can write
\begin{align}\label{eqn 6.6}
 W(\rho_1\otimes\rho_{m},\psi)=W(\rho_2\otimes\rho_{m},\psi)\implies\\\nonumber
\det(\rho_1)(\gamma)W(\rho_{m},\psi)^{\dim(\rho_1)}=
 \det(\rho_2)(\gamma)W(\rho_{m},\psi)^{\dim(\rho_2)}.
\end{align}
Again, using Theorem \ref{Theorem using Deligne-Henniart}, from Equation (\ref{eqn 6.6}), we have
\begin{equation}\label{eqn 6.7}
 \det(\rho_1)(\gamma)\cdot W(\chi_F,\psi)^{\dim(\rho_1\otimes\rho_m)}\cdot \det(\rho_0)(c^{\dim(\rho_1)})=
 \det(\rho_2)(\gamma)\cdot W(\chi_F,\psi)^{\dim(\rho_2\otimes\rho_m)}\cdot \det(\rho_0)(c^{\dim(\rho_2)}).
\end{equation}

Because $\det(\rho_1)\equiv\det(\rho_2)$ on $F^\times$, and $\chi_F$ is arbitrary character of $a(\chi_F)\ge 2$,
from the above equation (\ref{eqn 6.7}), we can conclude that $\dim(\rho_1)=\dim(\rho_2)$. 


This gives: \\
{\bf Case-1:} $\rho_1\cong \rho_2\otimes\mu$, where $\mu: G_F^\times\to \bbC^\times$ is a character of $G_F$,
hence $\mu$ can be considered a character of $F^\times$
(via class field theory).\\
{\bf Case-2:} $\rho_1\equiv\rho_2$, and this is one of our assertions.

When we are in Case-1,  by the given assumption 
$\det(\rho_1)\equiv\det(\rho_2)=\det(\rho_1\otimes\mu)=\det(\rho_1)\cdot \mu^{\dim(\rho_1)}$, then the order of $\mu$ 
must be a divisor of $\dim(\rho_1)=\dim(\rho_2).$

Now we are left to prove that $\mu$ is unramified, and which follows from the given condition
$W(\rho_1\otimes\rho_m,\psi)=W(\rho_2\otimes\rho_m,\psi)$.

This completes the proof.

 

\end{proof}



\section{{\bf Appendix}}

In the following lemma, we see an explicit description of the representation $\rho=\rho(X_\eta,\chi)$.

\begin{lem}[{\bf Explicit Lemma}]\label{Explicit Lemma}
 Let $\rho=\rho(X_\eta,\chi_K)$ be a U-isotropic Heisenberg representation of the absolute Galois group $G_F$ of a local field 
 $F/\bbQ_p$. Let $K=K_\eta$, and let $E/F$ be the maximal unramified subextension in $K/F$. Then
 \begin{enumerate}
  \item The norm map induces an isomorphism:
  $$N_{K/E}:K_F^\times/I_FK^\times\stackrel{\sim}{\to}I_FE^\times/I_F\cN_{K/E}.$$
  \item Let $c_{K/F}:F^\times/\rm{Rad}(X_\eta)\wedge F^\times/\rm{Rad}(X_\eta)\cong K_F^\times/I_FK^\times$ be the isomorphism
  which is induced by the commutator in the relative Weil-group $W_{K/F}$. Then, for units $\varepsilon\in U_F$, we 
  explicitly have:
  $$c_{K/F}(\varepsilon\wedge\pi_F)=N_{K/E}^{-1}(N_{E/F}^{-1}(\varepsilon)^{1-\varphi_{E/F}}),$$
  where $\varphi_{E/F}$ is the Frobenius automorphism for $E/F$ and where $N^{-1}$ means to take a preimage of the norm map.
  \item The restriction $\chi_K|_{K_F^\times}$ is characterized by:
  $$\chi_K\circ c_{K/F}(\varepsilon\wedge\pi_F)=X_\eta(\varepsilon,\pi_F)=\eta(\varepsilon),$$
  for all $\varepsilon\in U_F$, where $c_{K/F}(\varepsilon\wedge\pi_F)$ is explicitly given via (2).
 \end{enumerate}

\end{lem}

\begin{proof}
 {\bf (1).} Under the given conditions, we have: $K=K_\eta,$ and $K/F$ is the bicyclic extension 
 with $\rm{Rad}(X_\eta)=\cN_{K/F}$, and 
 $E/F$ is the maximal unramified subextension in $K/F$. Therefore, $K/E$ and $E/F$ both are cyclic, hence 
 $$E_F^\times=I_FE^\times,\qquad K_E^\times=I_EK^\times.$$
 From the diagram (3.6.1) on p. 41 of \cite{Z4}, we have 
 $$N_{K/E}: K_F^\times/I_FK^\times\stackrel{\sim}{\to} E_F^\times/I_F\cN_{K/E}.$$
 We also know that $E_F^\times=I_FE^\times$. Thus, the norm map $N_{K/E}$ induces an isomorphism:
 $$N_{K/E}:K_F^\times/I_FK^\times\cong I_FE^\times/I_F\cN_{K/E}.$$
 {\bf (2).} Under the given conditions, $c_{K/F}$ is the isomorphism induced by the commutator in 
 the relative Weil-group  $W_{K/F}$
 (cf. the map (\ref{eqn 5.1.3}). Here, $\rm{Rad}(X_\eta)=\cN_{K/F}=:N$.
 Then, from Proposition 1(iii) of \cite{Z5} on p. 128, we have 
 $$c_{K/F}: N\wedge F^\times/N\wedge N\stackrel{\sim}{\to} I_FK^\times/I_FK_F^\times$$
 as an isomorphism by the map:
 $$c_{K/F}(x\wedge y)=N_{K/F}^{-1}(x)^{1-\phi_F(y)},$$
 where $\phi_F(y)\in \rm{Gal}(K/F)$ for $y\in F^\times$ by class field theory.
 If $y=\pi_F$, then by class field theory (cf. \cite{JM}, p. 20, Theorem 1.1(a)), we can write 
 $\phi_F(\pi_F)|_{E}=\varphi_{E/F}$, where $\varphi_{E/F}$ is the Frobenius automorphism for $E/F$.
 
 Now, we come to our special case.
Because $E/F$ is unramified, we have $U_F\subset\cN_{E/F}$, and we obtain (cf. \cite{Z4}, pp. 46-47 of Section 4.4 and 
the diagram on p. 302 of \cite{Z2}):
\begin{equation}\label{eqn explicit lemma}
 N_{K/E}\circ c_{K/F}(\varepsilon\wedge\pi_F)=N_{E/F}^{-1}(\varepsilon)^{1-\varphi_{E/F}}.
\end{equation}
We also know (see the first two lines under the upper diagram on p. 302 of \cite{Z2}) that
$E_F^\times\subseteq \cN_{K/E}$. Here 
$$N_{E/F}^{-1}(\varepsilon)^{1-\varphi_{E/F}}\in I_FE^\times/I_F\cN_{K/E}=E_F^\times/I_F\cN_{K/E},$$
because $E/F$ is cyclic, hence $E_F^\times=I_FE^\times$. Therefore, from Equation (\ref{eqn explicit lemma}), we can conclude:
$$c_{K/F}(\varepsilon\wedge\pi_F)=N_{K/E}^{-1}(N_{E/F}^{-1}(\varepsilon)^{1-\varphi_{E/F}}).$$
{\bf (3.)} We know that $c_{K/F}(\varepsilon\wedge\pi_F)\in K_F^\times$, and $\chi_K:K^\times/I_FK^\times\to\bbC^\times$. 
Then, we can write 
\begin{align*}
 \chi_K\circ c_{K/F}(\varepsilon\wedge\pi_F)
 &=\chi_K(N_{K/E}^{-1}(N_{E/F}^{-1}(\varepsilon)^{1-\varphi_{E/F}})\\
 &=\chi_E\circ N_{K/E}(N_{K/E}^{-1}(N_{E/F}^{-1}(\varepsilon)^{1-\varphi_{E/F}}), \quad\text{since $\chi_K=\chi_E\circ N_{K/E}$}\\
 &=\chi_E(N_{E/F}^{-1}(\varepsilon)^{1-\varphi_{E/F}})=X_\eta(\varepsilon,\pi_F)\\
 &=\eta(\varepsilon).
\end{align*}
This is true for all $\varepsilon\in U_F$. Therefore, we can conclude that 
$\chi_K|_{K_F^\times}=\eta$.
\end{proof}

\begin{exm}[{\bf Explicit description of Heisenberg representations of dimension prime to $p$}]\label{Example for Heisenberg reps}

Let $F/\bbQ_p$ be a local field, and $G_F$ be the absolute Galois group of $F$.
Let $\rho=\rho(X,\chi_K)$ be a Heisenberg representation of $G_F$ of dimension $m$ prime to $p$. Then, from 
Lemma \ref{Lemma dimension equivalent},  the alternating character $X=X_\eta$ is U-isotropic for a character
$\eta:U_F/U_F^1\to\bbC^\times$. Here, from Lemma \ref{Lemma U-isotropic}, 
we can say $m=\sqrt{[F^\times:\rm{Rad}(X_\eta)]}=\#\eta$ divides $q_F-1$.

Because $U_F^1$ is a pro-p-group and $gcd(m,p)=1$, we have $(U_F^1)^m=U_F^1\subset {F^\times}^m$, and therefore  
$$F^\times/{F^\times}^m\cong\bbZ_m\times\bbZ_m,$$
is a bicyclic group of the order $m^2$. Therefore, by class field theory, there is precisely one extension $K/F$ such that 
$\rm{Gal}(K/F)\cong\bbZ_m\times\bbZ_m$, and the norm group $\cN_{K/F}:=N_{K/F}(K^\times)={F^\times}^m$.

We know that $U_F/U_F^1$ is a cyclic group of order $q_F-1$, hence $\widehat{U_F/U_F^1}\cong U_F/U_F^1$. By the given condition 
$m|(q_F-1)$, hence $U_F/U_F^1$ has exactly one subgroup of order $m$. Then, number of elements of order $m$ in $U_F/U_F^1$ is 
$\varphi(m)$, which is Euler's $\varphi$-function of $m$.
In this setting, we have $\eta\in \widehat{U_F/U_F^1}\cong \widehat{FF^\times/U_F^1\wedge U_F^1}$ with 
$\#\eta=m$. This implies that up to one-dimensional character twist there are $\varphi(m)$ representations 
corresponding to $X_\eta$ where $\eta:U_F/U_F^1\to\bbC^\times$ is of the order $m$.

\begin{rem}
 From this above explicit description of the Heisenberg representations of dimension prime to $p$, we can see that the 
 conductor of character $\eta$ is $a_F(\eta)=1$. From Equation (2.5) and Proposition 4.6, we can conclude
 that 
 \begin{equation}
  a_F(\eta)=j(\rho)+1.
 \end{equation}
Here, because $a_F(\eta)=1$, we have $j(\rho)=0$.
\end{rem}

According to Corollary 1.2 of \cite{Z2}, all dimension-m-Heisenberg 
representations of $G_F=\rm{Gal}(\overline{F}/F)$ is given as 
\begin{equation}
 \rho=\rho(X_\eta,\chi_K),\tag{1H}
\end{equation}
where $\chi_K: K^\times/ I_{F}K^\times\to\mathbb{C}^{\times}$ is a character 
such that the restriction of $\chi_K$
to the subgroup $K_{F}^{\times}$ corresponds to $X_\eta$ under the map (\ref{eqn 5.1.3}), and
\begin{equation}
 F^\times/{F^\times}^m\wedge F^\times/{F^\times}^m\cong K_{F}^{\times}/I_{F}K^\times,\tag{2H}
\end{equation}
which is given via the commutator in the relative Weil-group $W_{K/F}$ (for details arithmetic description of Heisenberg
representations of a Galois group, see \cite{Z2}, pp. 301-304).
Condition (2H) corresponds to (\ref{eqn 5.1.3}). Here, the above Explicit Lemma \ref{Explicit Lemma} comes in.

Due to our assumption, both sides of (2H) are groups of order $m$.
If one choice $\chi_K=\chi_0$ has been fixed, all other $\chi_K$
are given as
\begin{equation}\label{eqn 4.20}
 \chi_K=(\chi_F\circ N_{K/F})\cdot\chi_0,
\end{equation}
for arbitrary characters of $F^\times$. For an optimal choice $\chi_K=\chi_0$, and order of $\chi_0$,
we need the following lemma.

\begin{lem}\label{Lemma 5.3.3}
Let $K/F$ be the extension of $F/\bbQ_p$ for which $\rm{Gal}(K/F)=\bbZ_m\times\bbZ_m$. 
The $K_{F}^{\times}$ and $I_{F}K^\times$ are
as above. Then, 
 the sequence 
 \begin{equation}\label{eqn 4.21}
  1\to U_{K}^{1}K_{F}^{\times}/U_{K}^{1}I_{F}K^\times\to U_K/U_{K}^{1}I_{F}K^\times\xrightarrow{N_{K/F}} U_F/U_{F}^{1}\to
  U_F/U_F\cap {F^\times}^m\to 1
 \end{equation}
is exact, and the outer terms are both of order $m$, hence the inner terms are both cyclic of order $q_F-1$.
\end{lem}
\begin{proof}
 The sequence is exact because ${F^\times}^m=N_{K/F}(K^\times)$ is the group of norms, and 
 $F^\times/{F^\times}^m\cong \bbZ_m\times\bbZ_m$ implies
 that the right hand term\footnote{Since $gcd(m,p)=1$, we have 
 \begin{center}
  $U_F\cdot{F^\times}^m=(<\zeta>\times U_F^1)(<\pi_F^m>\times<\zeta^m>\times U_F^1)=<\pi_F^m>\times<\zeta>\times U_F^1$,
 \end{center}
where $\zeta$ is a $(q_F-1)$-st root of unity. 
Then, 
\begin{center}
 $U_F/U_F\cap {F^\times}^m=U_F\cdot {F^\times}^m/{F^\times}^m=
 <\pi_F^m>\times<\zeta>\times U_F^1/<\pi_F^m>\times<\zeta^m>\times U_F^1\cong\bbZ_m$.
\end{center}
Hence, $|U_F/U_F\cap{F^\times}^m|=m$.} is of the order $m$. By our assumption,
the order of $K_{F}^{\times}/I_{F}K^\times$ is $m$. Now, 
 we consider the exact sequence
 \begin{equation}\label{sequence 5.1.25}
  1\to U_{K}^{1}\cap K_{F}^{\times}/U_{K}^{1}\cap I_{F}K^\times\to K_{F}^{\times}/I_{F}K^\times\to 
  U_{K}^{1}K_{F}^{\times}/U_{K}^{1}I_{F}K^\times\to 1.
 \end{equation}
Since the middle term has order $m$, the left term must have order $1$, because $U_{K}^{1}$ is a pro-p-group and $gcd(m,p)=1$.
Hence, the right term is also of the order $m$. Therefore, the outer terms of the sequence (\ref{eqn 4.21}) 
have both orders $m$, hence, the inner 
terms must have the same order $q_F-1=[U_F:U_{F}^{1}]$, and they are cyclic, because the groups $U_F/U_{F}^{1}$ and $U_K/U_{K}^{1}$
are both cyclic.
\end{proof}

{\bf\large{We now are in a position to choose $\chi_K=\chi_0$ as follows}}: 
\begin{enumerate}
 \item we take $\chi_0$ as a character of $K^\times/U_{K}^{1}I_{F}K^\times$,
 \item we  take it on $U_{K}^{1}K_{F}^{\times}/U_{K}^{1}I_{F}K^\times$ as it is prescribed by the above 
 Explicit Lemma \ref{Explicit Lemma},
 in particular, $\chi_0$ restricted to that subgroup (which is cyclic of order $m$) will be faithful.
 \item we take it trivial on all primary components of the cyclic group $U_{K}/U_{K}^{1}I_{F}K^\times$ which are not $p_i$-primary,
 where $m=\prod_{i=1}^{n}p_i^{a_i}$.
 \item we take it trivial for a fixed prime element $\pi_K$.
\end{enumerate}

Under the above optimal choice of $\chi_0$, we have the following lemma.

\begin{lem}\label{Lemma 5.1.17}
Denote $\nu_p(n):=$ as the highest power of $p$ for which $p^{\nu_p(n)}|n$.
 The character $\chi_0$ must be a character of order 
 $$m_{q_F-1}:=\prod_{l|m}l^{\nu_l(q_F-1)},$$
 which we will call the $m$-primary part of $q_F-1$, so it determines a cyclic
extension $L/K$ of degree $m_{q_F-1}$ which is totally tamely ramified, and we can consider 
the Heisenberg representation $\rho=(X,\chi_0)$ of 
$G_F=\rm{Gal}(\overline{F}/F)$ is a representation of $\rm{Gal}(L/F)$, which is of order $m^2\cdot m_{q_F-1}$.
\end{lem}

\begin{proof}
Under the given conditions, $m|q_F-1$. Therefore, we can write
$$q_F-1=\prod_{l|m}l^{\nu_l(q_F-1)}\cdot \prod_{p|q_F-1,\; p\nmid m}p^{\nu_p(q_F-1)}=
m_{q_F-1}\cdot \prod_{p|q_F-1,\;p\nmid m}p^{\nu_p(q_F-1)},$$
where $l, p$ are prime, and $m_{q_F-1}=\prod_{l|m}l^{\nu_l(q_F-1)}$.

From the construction of $\chi_0$, $\pi_K\in\rm{Ker}(\chi_0)$, hence the order of $\chi_0$ comes from the restriction to 
$U_K$. Then, the order of $\chi_0$ is $m_{q_F-1}$, because from Lemma \ref{Lemma 5.3.3}, the order of $U_K/U_{K}^{1}I_FK$ is
$q_F-1$. Since order of $\chi_0$ is $m_{q_F-1}$, by class field theory $\chi_0$ determines a cyclic 
extension $L/K$ of degree $m_{q_F-1}$, hence 
$$N_{L/K}(L^\times)=\rm{Ker}(\chi_0)=\rm{Ker}(\rho).$$
This means $G_L$ is the kernel of $\rho(X,\chi_0)$, hence $\rho(X,\chi_0)$ is actually a representation of 
$G_F/G_L\cong\rm{Gal}(L/F)$.

Because $G_L$ is a normal subgroup of $G_F$,  $L/F$ is a normal extension of 
degree $[L:F]=[L:K]\cdot[K:F]=m_{q_F-1}\cdot m^2$.
Thus, $\rm{Gal}(L/F)$ is of the order $m^2\cdot m_{q_F-1}$.

Moreover, because $[L:K]=m_{q_F-1}$ and $gcd(m,p)=1$, $L/K$ is tame. By this construction, we have a prime 
$\pi_K\in\rm{Ker}(\chi_0)=N_{L/K}(L^\times)$, hence $L/K$ is a totally ramified extension. 

\end{proof}

\begin{lem}(Here $L$, $K$, and $F$ are the same as in Lemma \ref{Lemma 5.1.17})
 Let $F^{ab}/F$ be the maximal abelian extension. Then, we have 
$$L\supset L\cap F^{ab}\supset K\supset F, \quad\{1\}\subset G'\subset Z(G)\subset G=\rm{Gal}(L/F),$$
where $[L:L\cap F^{ab}]=|G'|=m$ and $[L:K]=|Z(G)|=m_{q_F-1}$.
\end{lem}

\begin{proof}
 Let $F^{ab}/F$ be the maximal abelian extension. Then, we have 
$$L\supset L\cap F^{ab}\supset K\supset F.$$
Here $L\cap F^{ab}/F$ is the maximal abelian in $L/F$. From the Galois theory, we can conclude 
$$\rm{Gal}(L/L\cap F^{ab})=[\rm{Gal}(L/F), \rm{Gal}(L/F)]=: G'.$$
Because $\rm{Gal}(L/F)=G_F/\rm{Ker}(\rho)$, and $[[G_F,G_F],G_F]\subseteq\rm{Ker}(\rho)$, 
from relation (\ref{eqn 5.1.3}) we have 
$$G'=[G_F,G_F]/\rm{Ker}(\rho)\cap [G_F,G_F]=[G_F,G_F]/[[G_F,G_F],G_F]\cong K_F^\times/I_FK^\times.$$
Again, from sequence \ref{sequence 5.1.25}, we have $|U_K^1K_F^\times/U_K^1 I_FK^\times|=|K_F^\times/I_FK^\times|=m$.
Hence $|G'|=m$.


From the Heisenberg property of $\rho$, we have 
$[[G_F,G_F],G_F]\subseteq\rm{Ker}(\rho)$, hence $\rm{Gal}(L/F)=G_F/\rm{Ker}(\rho)$ is a two-step nilpotent group.
This gives $[G',G]=1$, hence $G'\subseteq Z:=Z(G)$. Thus, $G/Z$ is abelian. 

Moreover, here $Z$ is the scalar group of $\rho$, hence the dimension of $\rho$ is:
$$\rm{dim}(\rho)=\sqrt{[G:Z]}=m$$
Therefore, the order of $Z$ is $m_{q_F-1}$ and $Z=\rm{Gal}(L/K)$.

\end{proof}

\begin{rem}[{\bf Special case: $m=2$, hence $p\ne 2$}]

Now, if we take $m=2$, hence $p\ne 2$, and choose $\chi_0$ as the above optimal choice, then we will have 
$m_{q_F-1}=2_{q_F-1}=2$-primary factor of the number $q_F-1$, and $\rm{Gal}(L/F)$ is a $2$-group of order 
$4\cdot 2_{q_F-1}$.

 When $q_F\equiv -1\pmod{4}$, $q_F$ is of the form $q_F=4l-1$, where $l\ge 1$. Therefore, we can write $q_F-1=2(2l-1)$.
Because $2l-1$ is always odd, when $q_F\equiv-1\pmod{4}$, the order of $\chi_0$ is $2_{q_F-1}=2$. 
Then, $\rm{Gal}(L/F)$ will be of order 8 if and only if $q_F\equiv -1\pmod{4}$, i.e., if and only
if $i\not\in F$. And if $q_F\equiv 1\pmod{4}$, then similarly,  we can write $q_F-1=4m$ for some integer $m\ge1$, hence 
$2_{q_F-1}\ge 4$. Therefore, when $q_F\equiv 1\pmod{4}$, the order of $\rm{Gal(L/F)}$ will be at least $16$.

\end{rem}

\end{exm}


{\bf Acknowledgment:} I express my sincere gratitude to E.-W. Zink for his valuable comments on the paper.

\end{document}